\documentclass[11pt, oneside]{article} %

\usepackage[margin=1in]{geometry}

\usepackage{graphicx}
\usepackage{multicol,multirow}
\usepackage{amsmath,amssymb,amsfonts}
\usepackage{mathrsfs}
\usepackage{amsthm}
\usepackage{rotating}
\usepackage{appendix}
\usepackage[numbers]{natbib}
\usepackage{ifpdf}
\usepackage[T1]{fontenc}
\usepackage{newtxtext}
\usepackage{newtxmath}
\usepackage{textcomp}
\usepackage{mathrsfs}
\usepackage{comment}
\usepackage{hyperref}

\newtheorem{theorem}{Theorem}[section]
\newtheorem{lemma}[theorem]{Lemma}
\newtheorem{proposition}[theorem]{Proposition}
\newtheorem{cor}[theorem]{Corollary}
\newtheorem{result}{Result}[section]
\newtheorem{assumption}{Assumption}[section]
\theoremstyle{definition}
\newtheorem{remark}[theorem]{Remark}
\newtheorem{example}[theorem]{Example}
\numberwithin{equation}{section}

\newcommand{\R}{\mathbb{R}}

\newcommand{\N}{\mathbb{N}}
\newcommand{\Z}{\mathbb{Z}}
\newcommand{\C}{\mathbb{C}}
\newcommand{\Sph}{\mathbb{S}}

\renewcommand{\AA}{\mathbf{A}}
\newcommand{\BB}{\mathbf{B}}

\newcommand{\LLambda}{\boldsymbol{\Lambda}}

\newcommand{\ord}{m}

\newcommand{\DiffOrd}{k_0}

\newcommand{\EE}[1]{\mathscr{E}_{#1}}

\newcommand{\MM}{\mathbf{M}}
\newcommand{\opL}{\mathcal{L}}

\newcommand{\Nn}{\mathcal{N}}

\newcommand{\diff}{\mathrm{d}}
\newcommand{\M}{\mathbb{M}}

\newcommand{\D}{N}
\renewcommand{\d}{d}

\DeclareMathOperator{\dist}{dist}

\renewcommand{\vec}[1]{\boldsymbol{\mathbf{#1}}}

\newcommand{\PhiB}{\mathbf{\Phi}}
\newcommand{\KK}{\mathbf{K}}

\newcommand{\Hom}{{\phi}^{\mathrm{S}}}
\newcommand{\tphi}{{\phi}_{\theta}^{\mathrm{M}}}
\newcommand{\tps}{{\phi}_{\theta}^{\mathrm{S}}}
\newcommand{\rbf}{\phi_{\theta}}
\newcommand{\kernel}{\Phi_{\theta}}

\newcommand{\LF}{\psi^{\flat} }
\newcommand{\HF}{\psi}

\newcommand{\ep}{\varepsilon}

\title{Kernel approximation beyond the native space -- with
applications to approximation on manifolds }

\author{T. Hangelbroek\thanks{Department of Mathematics,University of Hawai`i -- M\=anoa}, 
C. Rieger\thanks{Department of Mathematics and Computer Science, Philipps-Universit\"at Marburg}, 
G.B. Wright\thanks{Department of Mathematics, Boise State University}
}

\begin{document}
\maketitle

\abstract{This article treats kernel approximation and interpolation on embedded manifolds  of $\R^N$
using restrictions of positive and conditionally positive definite kernels. The main challenge is to develop an approximation theory that treats error measured in highly regular smoothness spaces 
relative to the kernel.  This means that the order of 
smoothness is higher than that of the kernel's native space: the reproducing kernel (semi-)Hilbert space generated by the kernel. This prevents the use of standard techniques for controlling error in this setting, especially reproducing kernel Hilbert space arguments like orthogonality of the interpolation projector, or bounds using the {\em power function}. 

To address this challenge, we extend methods for treating target functions given as potentials introduced by DeVore and Ron in
\cite{DevRon}, and give conditions guaranteeing that restrictions of such functions span Sobolev smoothness spaces
on the manifolds. Together with kernel-based Bernstein inequalities for embedded manifolds, these results give a comprehensive theory for interpolation error. 
As an application of the theory, we derive new error estimates for approximating the eigenspace of certain differential operators arising from kernel-based methods for partial differential equations on manifolds.
\let\thefootnote\relax\footnotetext{2020 \emph{Mathematics Subject Classification}. 65D12, 46E35, 47B34, 58J40.}
\let\thefootnote\relax\footnotetext{\emph{Key words and phrases}. kernel approximation, kernel interpolation, Mat{\' e}rn kernel, Besov spaces, Sobolev spaces.}
}

\section{Introduction}
We consider a positive definite kernel $\Phi$ on $\M$ having native space $\Nn(\Phi)$.
In case the native space is norm-equivalent to a Sobolev space,
we have the 
following basic native space interpolation estimate.
\begin{result} 
If $\Nn(\Phi)$ is embedded in $ H^{\sigma}(\M)$ then
for $f\in \Nn(\Phi)$ and $0\le s\le \sigma$, we have
$$
\|f-I_{\Xi}f \|_{H^{s}(\M)}
\le
 C h^{\sigma-s}
\|f-I_{\Xi}f \|_{H^{\sigma}(\M)}
\le 
Ch^{\sigma-s} 
\|f\|_{\Nn(\Phi)}.
$$
\end{result}
If $f=\int_{\M} \nu(z) \Phi(\cdot,z)\diff \mu(z)$, then $f\in \Nn(\Phi)$, and the identity
$\langle f, g\rangle_{\Nn(\Phi) }= \langle v,g\rangle_{L_2(\M)}$ holds. 
From this, we can obtain Schaback's doubling result, see \cite{schaback_doubling}:
\begin{result}
\label{doubling_result}
If $\Nn(\Phi)$ is embedded in 
$ H^{\sigma}(\M)$  and 
$f= 
\int_{\M} \nu(z) \Phi(\cdot,z)
\diff \mu(z)$ 
for some $\nu\in L_2(\M)$, then for 
$s\le \sigma$ we have
$$
\|f-I_{\Xi}f \|_{H^{s}(\M)}
\le
 C h^{2\sigma-s} \|\nu\|_{L_2(\M)}.
 $$
\end{result}
In general, the regularity of functions in the native space 
(i.e., the worst-case regularity)
will be of lower order than 
that of the generating functions  $\Phi(\cdot, z)$.
Thus, one may wonder if it is also possible to estimate $\|f-I_{\Xi}f\|_{H^{s}(\M)}$ for $s>\sigma$
when the target function is sufficiently smooth. 
This problem may be viewed as a new, complementary aspect of ``escaping the native space,'' 
a well-studied topic in kernel interpolation that has traditionally considered error estimates in lower-order Sobolev norms for rougher target functions, see e.g. \cite{Narcowich:etal:2006}.
It is motivated not only by the goal of developing a more comprehensive approximation theory for kernel interpolation, but also applications. For example,  
a key motivation
is the analysis of kernel-based numerical methods for the solution of partial differential equations (PDEs).

To address this problem, we consider kernel convergence
in  Sobolev norms of higher order than the native space.
The primary object in our analysis is the kernel's integral operator
$f\mapsto \int_{\M} f(y)\Phi(\cdot,y)\diff \mu(y)$.
This is sometimes called the  
{\em covariance operator}, see \cite{Stuart_Acta},
and it is 
central to a number of questions in machine learning and kernel approximation
\cite{CS,LGZ,SS}.
{Of particular importance are its mapping properties
\cite{Baker,Kovacs}
(i.e., regularity and range), which is a major focus of our investigation.}
The target functions which experience {\em doubling}, 
in the sense of \cite{schaback_doubling} 
and described above in Result \ref{doubling_result},
are precisely those
in the range of this operator.

Our approach is to consider, as in \cite{FW},
the above
problem on an embedded manifold $\M\subset \R^N$,
and to use kernels $\Phi(x,y) =\phi(x-y)$ which
are restrictions of nice,
translation invariant Euclidean kernels $\phi$
(called RBFs, discussed below). 
We study the integral operator of $\phi$, namely
$\nu\mapsto \int_{\R^N} \nu(y) \phi(\cdot-y)\diff \sigma( y)$
for certain compactly supported, finite, positive
measures  $\sigma$ and  for $\nu\in L_p(\sigma)$ (this includes
the case where $\sigma$ is the surface 
measure for a closed manifold $\M$),
and it
leads naturally 
to a novel
way to understand the kernel's integral operator:
as
convolution between $\phi$ and compactly supported
distributions.

\paragraph{Outline} 
We proceed as follows: in  section \ref{S_Background} we present background on Besov spaces, some basic results about conditionally positive definite kernels, and their native spaces used throughout
the paper.

Section \ref{S_RBF}
introduces the two families
of kernels used throughout the paper: the Mat{\'e}rn kernels, which are strictly positive definite
and the surface splines,
which are conditionally positive definite.
The object of this section is to study the kernel's 
integral operator, considered as a convolution operator on 
tempered distributions 
in case of Mat{\'e}rn kernels,
and on compactly supported
distributions in case of 
the surface splines.
We discuss the integral operator's 
mapping properties, 
showing that it acts 
continuously between Besov spaces.

Section \ref{S_potentials} gives  kernel approximation results 
for a scheme based on 
\cite{DevRon} 
for functions in the range of the integral operator and with 
error measured in Besov metrics on $\R^N$. 
Despite holding for Euclidean target functions (as opposed to target functions on the manifold), these results
are largely novel.

In section \ref{S_Main} we provide  results for kernel approximation
on an embedded manifold, including
\begin{itemize}
    \item error estimates for an approximation scheme based on an integral identity
    \item a higher order Bernstein (inverse) inequality for kernel spaces
\item new error estimates for kernel  interpolation
\item nontrivial examples where
$f\mapsto \int_{\M} f(y)\Phi(\cdot,y)\diff \mu(y)$
is a Sobolev isomorphism  (and, thus, where the range of this operator can be precisely
identified)
\item numerical results that illustrate the new error estimates.
\end{itemize}
Although the Bernstein inequality is interesting on its own, in the context of this article, its main use is to show that interpolation provides {\em near best} approximation in the high order Sobolev space $H^{s}(\M)$ for large $s$.

In Section \ref{S_eigenvalues} we apply the results from Section \ref{S_Main} to analyze the spectral properties of kernel differentiation by way of a Galerkin problems in high order Sobolev spaces. In particular, we give the first (to our knowledge) estimates for  approximating the spectrum and eigenspace of certain differential operators by   spectra and eigenspaces of kernel differentiation matrices.  This is important for analyzing the temporal stability of kernel methods for time-dependent parabolic problems.

Finally, an appendix considers the problem of approximating certain Hausdorff-like measures by discrete measures obtained from  local polynomial reproductions, with error measured
 in negative order Besov spaces. Although this does not involve kernels directly, it, along with the mapping properties discussed in section 3, provide  the main mechanism for 
estimating the error in section 4.

\section{Background}
\label{S_Background}

For a region $\Omega\subset \R^{\D}$, and a point set $\Xi\subset \Omega$, 
we define 
$$h(\Xi,\Omega) 
:= \max_{x\in\Omega}\min_{\xi\in\Xi} |x-\xi|
\quad \text{and} \quad q(\Xi) := \min_{\xi\in \Xi} \min_{\zeta\in\Xi\setminus\{\xi\}} |\xi-\zeta|.
$$
Similarly, for an embedded manifold $\M\subset \R^N$ with distance $\dist:\M\times \M\to [0,\infty)$ 
induced by the Riemannian metric, if $\Xi\subset \M$,  define $h(\Xi,\M) = \max_{x\in\M}\min_{x_j\in \Xi} \dist(x,x_j)$.
When clear from the context, we will simply use $h$ and $q$ in lieu of $h(\Xi,\Omega) $ and $q(\Xi)$.

Let  $\mathcal{P}_L(\R^{\D}) $  
denote the space of  polynomials 
of total degree $L$ or less in $\R^{\D}$.
For $\Omega\subset \R^{\D}$, define
 $$
 \mathcal{P}_{L}(\Omega)
 :=\bigl\{p\left|_{\Omega}\right.
 \mid 
 p\in  \mathcal{P}_{L}(\R^{\D})\bigr\},
 $$
  the set of restrictions of polynomials 
 to $\Omega$.
 Let $\mathcal{S}(\R^{\D})$ denote the space of Schwartz functions, and
 $\mathcal{S}'(\R^{\D})$, the space of tempered distributions.
Likewise, denote by $\mathcal{D}(\R^{\D})=C_c^{\infty}(\R^D)$ and 
$\mathcal{E}(\R^{\D})= C^{\infty}(\R^{\D})$ the spaces of (compactly supported, globally supported) test functions,
 by $\mathcal{D}'(\R^{\D})$ the standard space of distributions, and by $\mathcal{E}'(\R^{\D})$ the space of compactly supported distributions.
Finally, we set
$$\mathcal{P}_{L}(\R^N)^{\perp} = \bigl\{\mu \in \mathcal{S}'(\R^{\D})\mid 
\  \mu p = 0, \;\forall 
p\in \mathcal{P}_{L}(\R^N)
\bigr\}.$$

 For a Schwartz function 
 $\gamma\in \mathcal{S}(\R^{\D})$ 
 we define the Fourier transform by
 $$
 \mathcal{F}\gamma (\omega):=
 \widehat{\gamma}(\omega) :=(2\pi)^{-\D/2} \int_{\R^{\D}} \gamma(z) e^{-i\langle z,\omega\rangle} \diff \omega.
 $$
  For a tempered distribution $T$,  the Fourier transform is
  $\langle \mathcal{F}T,\gamma \rangle = \langle T, \mathcal{F} \gamma\rangle$.
  If the context is clear, we simply write $\widehat{T}=\mathcal{F}T$. 
  
  \subsection{Smoothness spaces}
  \label{SS_smoothness}
We briefly introduce (positive integer order) Sobolev spaces on $\R^N$, followed by Besov spaces, which 
include spaces of both integer and fractional order. (Although fractional order Sobolev spaces
are widely in use by incorporating {\em Slobodeckij} seminorms,  these are actually instances of Besov spaces.)
\subsubsection*{Sobolev spaces}
For $k\in\N$ and $p\in[1,\infty]$, we define the Sobolev space $W_p^k(\R^N)$ as a
Banach space of functions $f$ in $L_p(\R^N)$  having distributional derivatives $D^{\alpha}f \in L_p(\R^N)$ for
$|\alpha|\le k$, and equipped with the norm
$$\|f\|_{W_p^k(\R^N)} :=\left( \sum_{|\alpha|\le k} \|D^{\alpha} f\|_{L_p(\R^N)}^p\right)^{1/p}.$$

Although we also consider fractional order Sobolev spaces $W_p^{s}(\R^N)$ with $s\notin \Z$, these are examples
of Besov spaces defined below: $W_p^{s}(\R^N) = B_{p,p}^{s}(\R^N)$. (Unfortunately, for $p\neq 2$ and $s\in\Z$, the Sobolev
space is not a Besov space).

For a measure space $(M,\mathcal{S},\mu)$, we will use the abbreviation $L_{p}(\mu)$ to denote the norm $\|f\|^p_{L_{p}(\mu)}=\int_{M}|f(x)|^p \diff\mu(x)$ if the context is clear.

\subsubsection*{Besov spaces on $\R^N$}
 Let $\HF:\R^N\to [0,1]$ be a bandlimited (Schwarz) test function with annular support
$\mathrm{supp}(\widehat{\HF})\subset B(0,2)\setminus B(0,1/2)$ and which satisfies
$\sum_{k=-\infty}^{\infty} \widehat{\HF}(2^k \xi) =1$ for all $\xi\in\R^{N}\setminus\{0\}$.
Define 
\begin{equation}
\label{fourier_multiplier}
\HF_k:=2^{Nk} \HF(2^{k} \cdot)
\end{equation}
 to be a dyadic dilation of $\HF$ and  
 define $\LF \in \mathcal{S}(\R^N)$ by
 $\widehat{\LF} := 1-\sum_{k=1}^{\infty}\widehat{ \HF}(2^{-k}\cdot)$.

For $s\in \R$ and
$p,q\in[1,\infty]$, the Besov space $B_{p,q}^{s}(\R^N)$ 
consists of tempered distributions, $f$, for which $f*\LF\in L_p{(\R^N)}$ and each $\HF_k*f\in L_p{(\R^N)}$
and for which the norm
\begin{equation*}
\| f\|_{B_{p,q}^{s}(\R^N)}:= \|f*\LF\|_{L_p(\R^N)}+ \Bigl\|  k\mapsto 2^{ks} \|\HF_k*f\|_{L_p(\R^N)} \Bigr\|_{\ell_q(\N)}
\end{equation*}
is finite.  

\begin{remark}
For $s> 0$, we can replace $\|f*\LF\|_{L_p(\R^N)}$
by $ \|f\|_{L_p(\R^N)}$. Then 
$$\| f\|_{B_{p,q}^{s}(\R^N)}\sim 
 \|f \|_{L_p(\R^N)}+ \Bigl\|  k\mapsto 2^{ks} \|\psi_k*f\|_{L_p(\R^N)} \Bigr\|_{\ell_q(\N)}.
$$
\end{remark}

For a fixed $p\in[1,\infty]$, the  continuous embedding $B_{p,q_1}^{s_1}(\R^N) \subset B_{p,q_2}^{s_2}(\R^N)$  holds whenever
$s_1<s_2$ (regardless of $q_1,q_2$), or  when $s_1=s_2$ and $q_1\ge q_2$. For this reason, 
we introduce the order $\preceq$ on $\R \times[1,\infty]$ where
\begin{equation}
\label{order_def}
(s_1,q_1)\preceq (s_2,q_2) \qquad \iff \qquad 
(s_1<s_2) \quad\text{ or } \quad (s_1=s_2 \text{ and }q_1>q_2).
\end{equation}

\begin{remark}
\label{R_delta}
If $\mu$ is a finite measure, then
 $\|\mu*f\|_{L_p(\R^N)}
\le \|\mu\|_{TV} \|f\|_{L_p(\R^N)}$, 
so 
$\|\mu *\psi_k\|_{L_p(\R^N)}
\le 
2^{kN/p'}
\|\psi\|_{L_p(\R^N)}
\|\mu\|_{TV}$,
where $p'$ is the conjugate exponent to $p$, i.e., $1/p+1/p'=1$. 
Thus $\mu\in B_{p,\infty}^{-N/p'}(\R^N)$ and 
 \begin{equation}
 \label{finite_measure_besov}
 \|\mu\|_{B_{p,\infty}^{-N/p'} (\R^N)}\le C \|\mu\|_{TV} .
 \end{equation}
For $\mu= \delta$, $\|\delta*\psi_k
\|_{L_p(\R^N)} = 2^{kN/p'} \|\psi\|_{L_p(\R^N)}$, so $k\mapsto 2^{ks}\|\delta*\psi_k\|_{L_p(\R^N)}
\in \ell_q(\N)$ if and only if $s<-N/p'$ and $q\in [1,\infty]$, or in
 the extreme case $s=-N/p'$ and $q=\infty$. In other words, $\delta\in B_{p,q}^{s}(\R^N)$ if and only if
  $(s,q)\preceq(-N/p',\infty)$.
By (\ref{finite_measure_besov}),    $\|\sum_{\xi\in\Xi} a_{\xi} \delta_{\xi}\|_{B_{p,\infty}^{-N/p'} (\R^N)}\le C  \sum_{\xi\in\Xi} |a_{\xi}|$ holds.
  \end{remark}

For integer order Sobolev spaces 
with norm
$\|u\|_{W_p^k(\R^{\D})} := \sum_{|\alpha|\le k} \|D^{\alpha} u\|_{L_p}$
there is the embedding  $B_{p,p}^{k}(\R^N) \subset W_p^k(\R^N) \subset B_{p,2}^k(\R^N)$
when $p\ge2$ and the reverse  $B_{p,2}^{k}(\R^N) \subset W_p^k(\R^N) \subset B_{p,p}^k(\R^N)$ when $1\le p\le 2$.

 For fractional order Sobolev spaces of order $s$ with $k=\lfloor s\rfloor$ 
$$\|u\|_{W_p^{s}(\R^{\D})}^p := \|u\|_{W_p^{k}(\R^{\D})}^p
+
\sum_{|\alpha|=k}
 \int_{\R^N}\int_{\R^N}
 	\frac{|D^{\alpha}u(x)-D^{\alpha}u(y)|^p}{|x-y|^{{\D}+p(s-k)}}\,
\diff x\,\diff y
$$
we have $W_p^s(\R^N) = B_{p,p}^s(\R^N)$ for all $p\in [1,\infty]$.

\subsubsection*{Besov spaces on domains and characterization via moduli of smoothness}
For $\Omega\subset\R^N$, $s\in\R$, and $p,q\in [1,\infty]$,
we define 
$$\|f\|_{B_{p,q}^s(\Omega)} := \inf \bigl\{\|u\|_{B_{p,q}^s(\R^N)}\mid u|_{\Omega} =f\bigr\}.$$

For positive $s$ this can be treated with an  equivalent intrinsic norm:
for $j=0,1,\dots$, define 
$$\Delta_u^j f(x;\Omega) =\begin{cases}
 \sum_{k=0}^j (-1)^{j-k} \frac{j!}{(j-k)!k!}f(x+ku)&x+ku\in \Omega \text{ for }k=0\dots m,\\
 0&\text{ otherwise.}\end{cases}$$ 
For $f\in L_p(\Omega)$, 
we define the modulus of smoothness 
$\omega_p^j(f,\cdot;\Omega):(0,\infty)\to (0,\infty)$ as 
$$\omega_p^j(f,t;\Omega) := \sup \bigl\{\|\Delta_u^j f(\cdot;\Omega)\|_{L_p(\Omega)}\mid |u|\le t \bigr\} .$$
Note that $\omega_p^0(f,t;\Omega) =\|f\|_{L_p(\Omega)}$.
When $\Omega=\R^N$, we omit the final argument.
In other words, 
 $\Delta_u^j f(x)=\Delta_u^j f(x;\R^N)$ and  $\omega_p^j(f,t)
=\omega_p^j(f,t;\R^N)$.

This leads to the  norm:
$\|f\|_{L_p(\Omega)}+ \|t\mapsto t^{-s} \omega_p^j(f,t;\Omega)\|_{L_q(\frac{\diff t}{t})}$,
where the $L_q$ norm employs the Haar measure $\frac{\diff t}{t}$ on $(0,\infty)$.
For $\Omega=\R^N$,
this is equivalent to $\|f\|_{B_{p,q}^s(\R^N)}$ when  $0<s<j$ (see for instance \cite[Section 2.6.1]{Trieb} or \cite[Theorem 6.25]{BL}),
while for $\Omega$ with smooth boundary, 
it is equivalent to 
$\inf \{\|u\|_{B_{p,q}^s(\R^N)}\mid u|_{\Omega} =f\}$.
From now on, we will always assume $0<s<j$ 
and use the fact that different values of $j$ 
will result in equivalent norms.
We denote the Besov seminorm by
$$|f|_{B_{p,q}^s(\R^N)}:=\|t\mapsto t^{-s} \omega_p^j(f,t;\Omega)\|_{L_q(\frac{\diff t}{t})}.$$
Note that for any integers $r,k\ge 0$,
\begin{align}
\omega_p^{r+k}(f,t;\Omega) &\le 2^k \omega_p^{r}(f,t;\Omega)  \label{mod_smoothness_order}\\
	\omega_p^{r+k}(f,t;\Omega) &\le  N^{k/2} t^k\max_{|\alpha|=k} \omega_p^{r}(D^{\alpha} f,t;\Omega) 
	\label{mod_smoothness_deriv}
\end{align}
These can be found in \cite[Section 2]{JS}.

 \subsection{Positive and conditionally positive definite kernels}
 \label{SS_PD_kernels}
 For a  subset $\Omega\subset \R^N$, 
 a kernel $\Phi:\Omega\times\Omega\to \R$
 in $C(\Omega\times\Omega)$
 is {\em conditionally positive definite}  (CPD) of order $\ord$
 if for any finite set
 $\Xi\subset \Omega$, 
 the collocation matrix 
 $\Bigl(\Phi(\xi,\zeta)\Bigr)_{\xi,\zeta\in \Xi}$ 
 is positive definite on the 
 subspace $\Lambda_{\Xi,\ord}\subset\R^{\Xi}$, where
$$\Lambda_{\Xi,\ord}:=\Bigl\{a:\Xi\to \R\mid (\forall q\in  \mathcal{P}_{\ord-1}(\Omega))\, \sum_{\xi\in \Xi} a_{\xi}q(\xi)=0\Bigr\}.$$
We define the $\#\Xi$ dimensional (trial) space as
$$V_{\Xi,\ord}(\Phi)
:=
\Bigl\{
\sum_{\xi \in \Xi}a_{\xi}\Phi(\cdot,\xi)
\mid  a\in \Lambda_{\Xi,\ord}
\Bigr\}+ \mathcal{P}_{\ord-1}(\Omega).$$
Note that if $\Phi$ is CPD of order $\ord$ then it is CPD of order $\ord+1$ as well, however
the spaces 
$V_{\Xi,\ord}(\Phi)$ and $V_{\Xi,\ord+1}(\Phi)$ will differ somewhat.

If $\ord\le 0$, then $ \mathcal{P}_{\ord-1}(\Omega)=\{0\}$.
In this case,  $\Phi$ is {\em  positive definite} (PD), and we simply have 
$V_{\Xi}(\Phi) := V_{\Xi,0}(\Phi) =\mathrm{span}_{\xi\in \Xi}\Phi(\cdot,\xi)$.

\subsection{Native spaces}
For a CPD kernel of order $\ord$ 
there is an associated function space, 
called the native space, $\mathcal{N}_{\ord}(\Phi)$, 
which consists of continuous functions. 
The space has the following properties:
\begin{itemize}
\item there is a semi-inner product (i.e., symmetric, non-negative bilinear form)
$(f, g) \mapsto \langle f, g\rangle_{\mathcal{N}_{\ord}(\Phi)}$
 with nullspace $\mathcal{P}_{\ord-1}(\Omega)$
\item  $\mathcal{N}_{\ord}(\Phi)/\mathcal{P}_{\ord-1}(\Omega)$ is a Hilbert space
\item for any finite set $\Xi\subset \Omega$, $V_{\Xi,\ord}\subset \mathcal{N}_{\ord}(\Phi)$
\item  for any finite measure $\sum_{\xi\in \Xi} a_{\xi} \delta_{\xi}$ with coefficients $(a_{\xi})\in \Lambda_{\Xi,\ord}$,
the representation formula 
\begin{equation}
\label{rep_formula}
\sum_{\xi\in \Xi} a_{\xi} f(\xi) =\bigl\langle f, \sum_{\xi\in\Xi} a_{\xi} \Phi(\cdot,\xi)\bigr\rangle_{\mathcal{N}_{\ord}(\Phi)}
\end{equation}
holds.
\end{itemize}

  We denote the induced seminorm by $f \mapsto \|f\|_{\mathcal{N}_{\ord}(\Phi)}$.
If $\Phi$ is positive definite, the native space is a Hilbert space and $f \mapsto\|f\|_{\mathcal{N}_{\ord} (\Phi)}$ is a norm.

The set of  finitely supported  measures which annihilate  $\mathcal{P}_{\ord-1}(\Omega)$
forms a dense subspace of $\bigl(\mathcal{N}_m(\Phi)/\mathcal{P}_{\ord-1}(\Omega)\bigr)'$.
If  $\mu=\sum_{\xi\in \Xi} a_{\xi} \delta_{\xi}$
for some $\Xi\subset \Omega$   with coefficients $(a_{\xi})\in \Lambda_{\Xi,\ord}$
then $\mu\perp  \mathcal{P}_{\ord-1}(\Omega)$ 
and  $\mu\in (\mathcal{N}_{\ord}(\Phi)/\mathcal{P}_{\ord-1}(\Omega))'$
 has
norm $\|\mu\|_{\mathcal{N}_m(\Phi)'} = \sqrt{ \sum_{\xi,\zeta\in \Xi} a_{\xi} a_{\zeta} \Phi(\xi,\zeta)}$
(n.b., this is independent of $\ord$).

The space 
$ \mathcal{P}_{\ord-1}(\Omega)^{\perp}
:=\{\mu\mid \exists \Xi\subset \Omega,\  \mu=\sum_{\xi\in \Xi} a_{\xi} \delta_{\xi},\ (a_{\xi})\in \Lambda_{\Xi,\ord}\}$
of such measures is dense in $ (\mathcal{N}_{\ord}(\Phi)/\mathcal{P}_{\ord-1}(\Omega))'$;
this yields a test for membership in $\mathcal{N}_{\ord}(\Phi)$:
$f\in \mathcal{N}_{\ord}(\Phi)$ iff 
there is a constant $C$ so that for every finite measure 
$\mu=\sum_{\xi\in \Xi} a_{\xi} \delta_{\xi}$ with coefficients $(a_{\xi})\in \Lambda_{\Xi,\ord}$,
\begin{equation}
\label{membership}
|\mu f |\le C\|\mu\|_{\mathcal{N}_m(\Phi)'}. 
\end{equation}
Furthermore, 
$
\|f\|_{\mathcal{N}_{\ord}(\Phi)} 
= 
\sup\frac{|\mu f|}{\|\mu\|_{\mathcal{N}_m(\Phi)'}} 
= 
\inf \{C \mid (\ref{membership}) \text{ holds for all  }
\mu\in \mathcal{P}_{\ord-1}(\Omega)^{\perp}
\}$.

{Thus we have the simple inclusion $\mathcal{N}_{\ord}(\Phi)\subset \mathcal{N}_{\ord+1}(\Phi)$. 
This is because 
$ \mathcal{P}_{\ord}(\Omega)^{\perp}
\subset \mathcal{P}_{\ord-1}(\Omega)^{\perp}$.
So if 
$|\mu f |
\le 
C\|\mu\|_{\mathcal{N}_{\ord}(\Phi)'}$ 
holds for all 
$\mu\in \mathcal{P}_{\ord-1}(\Omega)^{\perp}$,
it must hold for all $\mu \in  \mathcal{P}_{\ord}(\Omega)^{\perp}$.
}

{Any 
element $f+\mathcal{P}_{\ord}(\Omega)$ 
of $\mathcal{N}_{\ord+1}(\Phi)/  \mathcal{P}_{\ord}(\Omega)$ which is orthogonal to 
$\mathcal{N}_{\ord}(\Phi)/ \mathcal{P}_{\ord}(\Omega)$
must, by the reproduction property,
 satisfy $\mu f =0$ for any $\mu\in \mathcal{P}_{\ord-1}(\Omega)^\perp$, and is thus trivial.
It follows that $\mathcal{N}_{\ord+1}(\Phi)=\mathcal{N}_{\ord}(\Phi)+\mathcal{P}_{\ord}(\Omega)$.}

{Note that the native space depends both on the kernel $\Phi$ and the order $\ord$; 
this is relevant because of the nesting property described above, 
so a given CPD function will generate infinitely many native spaces (one for each order).}

\section{The RBF and   its  kernel  integral operator}
\label{S_RBF}
We consider  radial functions $\rbf\in C(\R^N)$
 on
$\R^N$
of two types: 
positive definite {\em Mat{\'e}rn kernels} and conditionally positive definite {\em surface splines}. 
Both have the advantage of having a very simple Fourier transform, and for even integer orders they are Green's functions for 
elliptic partial differential operators.

\subsection{The Mat{\'e}rn and surface spline kernels and their parameters}
In general, we will denote  the kernel  $(x,y)\mapsto \rbf(x-y)$ by $\rbf$ as well;
 the meaning will be clear from context.
When needed, we refer to the Mat{\'e}rn kernel by $\tphi$ and
the surface spline kernel by $\tps$.
The kernels from either family are associated with 
a local {\em order} $\theta$ corresponding to the behavior of $\rbf(x,z)$ near to the diagonal 
$\{(x,z)\in \R^N\times \R^N\mid x=z\}$.
This is a perhaps not the most conventional choice of parametrization. 
Two common alternatives are 
\begin{itemize}
\item
the parameter  $\tau=\theta+N$, which corresponds to the 
the decay of the  $N$-dimensional Fourier transform $\widehat{\rbf}(\xi) \sim |\xi|^{-(\theta+N)}$, and
thus to the  order of the pseudodifferential operator which $\rbf$ inverts.
\item the parameter
$\tau/2 = (\theta+N)/2$, which is the Sobolev order of the {\em native space} $\mathcal{N}_m(\rbf)$
in the sense that $W_2^{\tau/2}(\R^N)\subset \mathcal{N}_{\ord}(\phi)\subset W_{2,loc}^{\tau/2}(\R^N)$.
\end{itemize}

The two kernels are both members of  the larger class considered in \cite[Assumption 2]{HR-Extending},
which also includes a number of compactly supported RBFs; in order to streamline the presentation and to avoid some issues
with the singular support of the RBF, 
we restrict our  focus to these two families, which are widely in use in applications.

\subsubsection*{Mat{\'e}rn kernels} 
The Mat{\'e}rn kernel, which is strictly positive definite, is denoted $\tphi$.
 Using a modified Bessel function (see \cite[Chapter 9.6]{AS}), 
 $$\tphi(x) =  |x|^{\theta/2} K_{\theta/2}(|x|),$$
where $\theta>0$. It has Fourier transform 
\begin{equation}
\label{mat_FT}
 \widehat{\tphi} (\xi)= \beta_{\tau,N} (1+|\xi|^2)^{-(\theta+N)/2}.
 \end{equation}
 It follows (see \cite[Corollary 10.13]{Wendland_book}) that the native space 
 $\mathcal{N}(\tphi)$ is
 $W_2^{\frac{\theta+N}{2}}(\R^N)$ .

\subsubsection*{Surface splines} 
For order $\theta>0$, the surface spline kernel is
defined by
$$\Hom_\theta(x)=\begin{cases} |x|^\theta& \theta\notin 2\N\\ |x|^\theta\log|x| & \theta\in 2\N.\end{cases}$$
This kernel is conditionally positive definite of order $\ord = 
\lfloor\frac{\theta}{2}\rfloor+1$.
This follows from 
the precise characterization of CPD order for RBFs given in \cite[Theorem 2.1]{Guo:etal:1993}; 
for the surface spline kernels this has been worked out
in \cite[Examples  9.2 and 9.3]{Fasshauer}.

The distributional Fourier transform of $\Hom_{\theta}$ on $\R^N\setminus\{0\}$ satisfies, 
for a positive constant $\alpha_{\theta,N}$,
\begin{equation}
\label{hom_FT}
\widehat{\Hom_{\theta}}(\xi) =
\alpha_{\theta,N} |\xi|^{-\theta-N}.
\end{equation}
This follows from \cite[Theorem 8.16]{Wendland_book}  when $\theta\notin 2\N$ and \cite[Theorem 8.17]{Wendland_book} for $\theta\in 2\N$
(see these results for the precise value of $\alpha_{\theta,N}$).
Note that (\ref{hom_FT}) does not extend to a locally $L_1$ function on $\R^N$. 
However, for any measure $\mu$ 
which has 
$\widehat{\mu}(\xi) 
= 
\mathcal{O}(|\xi|^{\lfloor \theta/2\rfloor +1})$ as $|\xi|\to 0$, 
the function 
$\xi\mapsto |\widehat{\mu}(\xi)|^2 
|\xi|^{-\theta-N}$
is locally integrable; 
this is guaranteed when 
$\mu\perp \mathcal{P}_{\lfloor \theta/2\rfloor}$.
In short, the
fact that the CPD order $m$
of $\tps$ must satisfy $m\ge \lfloor \theta/2\rfloor+1$ guarantees that 
$\mathcal{N}_{\ord}(\tps)$
consists of (equivalence classes of) functions $f+\mathcal{P}_{m-1}(\R^N)$ 
for which $\int_{\R^d}|\widehat{f}(\xi)|^2 
|\xi|^{-\theta-N}\diff \xi$ converges (see 
also \cite[Theorem 10.21]{Wendland_book}).

\subsection{Kernel integral operator for the Mat{\'e}rn kernel}
We now consider a generalization of the integral operator 
$$\rbf*: g\mapsto \int_{\R^N} \rbf(\cdot-y)g(y)\diff y$$
to treat distributions in Besov spaces (possibly of negative order) with the goal of understanding 
the operator's mapping properties. 
This works cleanly for the Mat{\' e}rn kernel, but requires some technical modifications for the 
surface spline kernels.

The Mat{\'e}rn kernel is notable because it is, modulo the  constant factor 
$ \beta_{\tau,N} $ given in (\ref{mat_FT}),
the integral kernel for the Bessel potential operator 
 $\mathcal{J}^{-(\theta+N)}$,
 where 
 $\mathcal{J}^{t}: \mathcal{S'} \to \mathcal{S'}: f\mapsto \mathcal{F}^{-1} \Bigl(\bigl(1+|\cdot|^2)^{t/2} \mathcal{F} f\bigr)$.
By \cite[Theorem 6.2.7]{BL}, $\mathcal{J}^{t}$ is an isomorphism between 
$B_{p,q}^s (\R^N)$
and $B_{p,q}^{s-t}(\R^N)$ for any $s\in\R$, $p,q\in[1,\infty]$.
Thus, convolution with the Mat{\'e}rn kernel
 \begin{equation}
 \label{Matern_Isomorphism}
 \tphi*:B_{p,q}^s (\R^N)\to B_{p,q}^{s+\theta+N}(\R^N): f\mapsto \tphi*f
 \end{equation}
is a Banach space isomorphism.

 By Remark \ref{R_delta}, 
 $\delta \in B_{p,q}^{\tilde{s}}(\R^N)$ iff $({\tilde{s}},q)\preceq (N/p',\infty)$, 
 so it follows that  for any $m\in\N$ and any finite $\Xi\subset \R^N$,
 $V_{\Xi,m}(\tphi)$,
 and thus
 $\tphi =\tphi*\delta \propto  \mathcal{J}^{-(\theta+N)}\delta$ itself,
 is  contained in $B_{p,q}^{s}(\R^N)$ if and only if $(s,q)\preceq (\theta+N/p,\infty)$.

 \subsection{Kernel integral operator for the surface splines} 
Surface spline convolution $\tps*$ is more technically complicated than  Mat{\'e}rn convolution, 
but there is a  mapping result comparable to (\ref{Matern_Isomorphism}).
In this way, it is connected to the generalize {\em Riesz potential} operator $\mathcal{I}^t: f\mapsto  \mathcal{F}^{-1} \bigl(|\cdot|^t \mathcal{F} f\bigr)$
defined on suitable tempered distributions.

\begin{lemma} 
\label{L_TPS_seminorm}
If $\mu\in \mathcal{E}'\cap B_{p,q}^s(\R^N)$ 
then 
 $$
\bigl\| k\mapsto 2^{k(\theta+N)}\|(\tps*\mu)*\HF_k\|_{L_p}\bigr\|_q
\le C \|\mu\|_{B_{p,q}^s(\R^N)}.$$
\end{lemma}
\begin{proof}
Because $\mu$ is compactly supported,
$g \mapsto \mu*g$ is a continuous operator on $\mathcal{S}(\R^N)$,
and so $\tps*\mu\in \mathcal{S}'(\R^N)$.

Thus $(\tps*\mu)*\psi_k=\tps*(\mu*\psi_k)$, and since $\mathcal{F}(\mu*\psi_k)$ 
vanishes in a neighborhood of the origin, 
$$
\mathcal{F}\bigl((\tps*\mu)*\phi_k) \bigr)
= \alpha_{\theta,N}
|\cdot|^{-(\theta+N)}\mathcal{F}(\mu *\psi_k) 
= \alpha_{\theta,N} \mathcal{F}(\mathcal{I}^{-(\theta+N)}\mu*\psi_k).$$
holds.  
The estimate \cite[Lemma 6.2.1]{BL} shows that
$\|\mathcal{I}^{-(\theta+N)}\psi_k* \mu\|_p \le C 2^{-(\theta+N)k} \|\psi_k*\mu\|_p$, so 
$\|   (\tps*\mu)*\psi_k\|_{L_p} \le C 2^{-(\theta+N)k} \|\psi_k*\mu\|_p$
for every $k\in\N$, and the result follows.
\end{proof}
Getting a result similar to 
(\ref{Matern_Isomorphism}),
 is complicated by the
fact that $\tps$ does not decay.
Even for compactly supported 
distributions, 
$\tps*\mu$ will not be in 
$B_{p,q}^{s+\theta+N}(\R^N)$;
it will not decay rapidly enough to 
ensure that $\tps*\mu*\LF\in L_p(\R^N)$.

One way to treat this is
to develop an estimate on a bounded 
domain $\Omega$: i.e. an estimate of the form
\begin{equation}
    \label{TPS_mapping}
\|\tps*\mu\|_{ B_{p,q}^{s+\theta+N}(\Omega)}\le C\|\mu\|_{ B_{p,q}^{s}(\R^N)}.
\end{equation}
This aligns nicely with results for pseudodifferential operators, but is also not difficult to prove from scratch.
Ultimately, it is unnecessary
for the purposes of our later results.

Instead, one may  consider extra 
{\em side conditions} on $\mu$ 
that guarantee 
$\tps*\mu*\LF\in L_p(\R^N)$; this
is done in Lemma \ref{L_side_condition} below.
To do so,
 we employ the following  observations.
\begin{lemma} 
If $\mathrm{supp}(\mu)$ is compact and compactly contained in a compact set $ K$,
then for $s\in\R$ and $p,q\in[1,\infty]$
there exist an order  $k_0$  depending on $s$
and a constant   $C$ depending on $s,p,q$ and $K$, so that
 for all $g\in C^{\infty}(\R^N)$,
\begin{equation}
\label{finite_order}
|\langle \mu , g\rangle | \le C \|\mu\|_{B_{p,q}^s(\R^N)}
\|g\|_{C^{k_0}(K)}
\end{equation}
holds.
\end{lemma} 
\begin{proof}
The inequality follows from the estimate
$|\langle \mu , g\rangle | \le C \|\mu\|_{B_{p,q}^s(\R^N)} \|g\|_{B_{p',q'}^{-s}(\R^N)} $
which holds for any configuration of $s$ and $p,q\in [1,\infty]$ (even when $p,q=\infty$).
Furthermore, we may assume without loss, that $g$ has compact support in $K$ 
by the following argument: if $\tau$ is a smooth cutoff with support in $K$ but equaling 1 on 
$\mathrm{supp}(\mu)$,
then $\tilde{g} = \tau g$ and $g$ satisfy 
$\langle \mu,g\rangle = \langle \mu,\tilde{g}\rangle$.
For a constant $C$ independent of $g$, we have, by
 $\|\tilde{g}\|_{C^{k_0}(K)}\le C \|g\|_{C^{k_0}(K)}$ by the product rule.
For $k_0>-s+1$ we have    
$\|g\|_{B_{p',q'}^{-s}(\R^N)} \le  \|g\|_{B_{p',q'}^{k_0-1}(\R^N)}  $. 
We can estimate this norm by
$$\|g\|_{B_{p',q'}^{k_0-1} (\R^N)} \le 
C(\|g\|_{L_p(K)} + \|t\mapsto t^{1-k_0} \omega_p^{k_0}(g,t;K)\|_{L_q(\frac{\diff t}{t})}).$$
By the estimate (\ref{mod_smoothness_deriv}), we have  
$ \omega_p^{k_0}(g,t;K)\le Ct^{k_0}  \max_{|\alpha| =k_0}\|D^{\alpha}g \|_{L_p(K)}$.
We also have the inequality
 $ \omega_p^{k_0}(D^{s}g,t;K)\le Ct^{k_0-1}  
 \max_{|\alpha| =k_0-1} \omega_p^1(D^{\alpha} g,t,K)
 \le Ct^{k_0-1} 
 \max_{|\alpha| =k_0-1}\|D^{\alpha}g \|_{L_p}$
 by (\ref{mod_smoothness_deriv}) followed by (\ref{mod_smoothness_order}).
 Using these two estimates, we derive the inequality
 $$ \|t\mapsto t^{1-k_0} \omega_p^{k_0}(g,t;K)\|_{L_q(\frac{\diff t}{t})} \le C  \max_{|\alpha| =k_0-1,k_0}\|D^{\alpha}g \|_{L_p(K)}.$$
Since  $\|D^{\alpha} g\|_{L_p(K)}\le \mathrm{vol}(K)^{1/p} \max_{x\in K} |D^{\alpha} g(x)|$, 
the result follows.
\end{proof}

Having $\mu$ which
annihilates polynomials guarantees
the convolution decays, and this
is the setting we work in.
The following result 
gives  side conditions  guaranteeing inclusion
in the global Besov space $B_{p,q}^{s+\theta+N}(\R^N)$.
\begin{lemma}
\label{L_side_condition}
If $\mu\in B_{p,q}^{s}(\R^N)\cap  \mathcal{E}'$  and $\mu\perp  \mathcal{P}_{L-1}$ with $L>\theta+N/p$,
then $\tps*\mu \in B_{p,q}^{s+\theta+N}(\R^N)$.
For compact $K\subset \R^N$ there is $C$ so that if $\mathrm{supp}(\mu)\subset K$, then 
$$\|\tps*\mu\|_{B_{p,q}^{s+\theta+N}(\R^N)}\le C \|\mu\|_{B_{p,q}^{s}(\R^N)}.$$
\end{lemma}
\begin{proof}
Denote by $K$ a compact set containing the support of $\mu$.
It follows  that  $\tps*\mu$ is smooth in  $\R^N\setminus K$, 
and for $x\notin K$, the convolution can be written as $|\tps*\mu(x)| = |\langle \mu ,\tps( x-\cdot)\rangle| $.
Furthermore, because  $\mu$ annihilates polynomials in $ \mathcal{P}_{L-1}$, we have
$ |\langle \mu ,\tps( x-\cdot)\rangle| =  |\langle \mu ,\tps( x-\cdot)-p\rangle|$ 
for any $p\in \mathcal{P}_{L-1}$. Thus,
$$|\tps*\mu(x)| 
\le C \|\mu\|_{B_{p,q}^s} \sup_{|\alpha|\le k_0}  \sup_{y\in K} |D^{\alpha} (\tps(y-x)- p(y))|$$ 
holds for any 
polynomial $p\in \mathcal{P}_L(\R^N)$, since $\mu\perp \mathcal{P}_{L-1}$.
Applying Taylor's remainder formula gives
\begin{equation}
\label{TPS_convolution_global_decay}
|\tps*\mu(x)|
\le C  \|\mu\|_{B_{p,q}^s(\R^N)} (\dist(x,K))^{\theta-L}
\end{equation}
for  $x\notin K$.
Consequently, for a smooth cutoff function $\tau\in C_c^{\infty}(\R^N)$, with $\tau =1$ on a neighborhood of $K$,
$(\tps*\mu )(1-\tau)\in L_{p}(\R^N)$, and by integrating (\ref{TPS_convolution_global_decay}), 
we have $\|(\tps*\mu )(1-\tau)\|_{L_p} \le C  \|\mu\|_{B_{p,q}^s} $.
By Young's inequality, it follows that
$(\tps*\mu )(1-\tau)*\LF\in L_p(\R^N)$ as well, with 
\begin{equation}
\label{TPS_Lp_far}\|(\tps*\mu )(1-\tau)*\LF\|_{L_p(\R^N)} \le C \|\LF\|_{L_1(\R^N)} \|\mu\|_{B_{p,q}^s(\R^N)} .
\end{equation}
Since
$(\tps*\mu )\tau*\LF$ is a Schwartz function, $\tps*\mu*\LF\in L_p(\Omega)$. 
 Write $\tilde{\LF}_x:=\LF(x-\cdot)$ and observe that 
 $$(\tps*\mu )\tau*\LF(x) = \langle (\tps*\mu )\tau, \tilde{\LF}_x\rangle =  \langle \tps*\mu , \tau \tilde{\LF}_x\rangle = 
  \langle \mu ,\tps*( \tau \tilde{\LF}_x)\rangle $$
  (where the last equality uses the fact that $\tps$ is even).
  Using the order estimate (\ref{finite_order}) of $\mu$, we have
  $$|(\tps*\mu )\tau*\LF(x) |\le C \|\mu\|_{B_{p,q}^{s}} \max_{y\in K}\max_{|\alpha|\le k_0}|D^{\alpha}\bigl( \tps*( \tau \tilde{\LF}_x)\bigr)(y) |.$$
  The final quantity is bounded and has fast decay:  one sees this by direct estimation of the integral; 
  to this end, observe that  $\tilde{K}:=K-K=\{z_1-z_2\in\R^N\mid z_1,z_2\in K\}$  is a compact set and
  \begin{eqnarray*}
  \max_{y\in K}\ |\bigl( \tps*D^{\alpha}( \tau \tilde{\LF}_x)\bigr)(y) |
  &=&
   \max_{y\in K}\Bigl|\int_{\tilde{K}} \tps(z) D^{\alpha}( \tau \tilde{\LF}_x)\bigr)(y-z)\diff z\Bigr|\\
 &\le & \mathrm{vol}(\tilde{K}) \, \max_{z\in \tilde{K}}|\tps(z) | \,\max_{y\in K} | D^{\alpha}( \tau \tilde{\LF}_x)\bigr)(y)|
 \end{eqnarray*}
 Since $\LF\in \mathcal{S}(\R^N)$, there is a constant $\Gamma$ depending on $\phi, \tau,  K,k_0$   obtained by manipulating Schwarz semi-norms  so that for all $|\beta|\le k_0$ and all $z\in K$,
 $\max_{y\in K}|D^{\beta} \LF(x -y)|\le \Gamma (1+|x|)^{-N-1}$.
It follows that, for an enlarged a constant $C$ which is independent of $x$
  $$ \max_{y\in K}\max_{|\alpha|\le k_0}|D^{\alpha}\bigl( \tps*( \tau \tilde{\LF}_x)\bigr)(y)| \le \Gamma \, \mathrm{vol}(\tilde{K}) \, \max_{z\in \tilde{K}}|\tps(z) | (1+|x|)^{-N-1}$$
  which implies  (by enlarging the constant) that $\|(\tps*\mu )\tau*\LF\|_{L_p(\R^N)}\le C \|\mu\|_{B_{p,q}^{s}(\R^N)} $ holds with $C$ independent of $\mu$;
  combining this with (\ref{TPS_Lp_far}) yields
  $ \|(\tps*\mu)*\LF \|_{L_p(\R^N)}\le C\|\mu\|_{B_{p,q}^{s}(\R^N)}$.
  The lemma follows by combining this with Lemma \ref{L_TPS_seminorm}.
\end{proof}

\section{Kernel approximation of potentials on  
\texorpdfstring{$\R^N$}{}}
\label{S_potentials}
In this section we consider the approximation by kernels $\phi$ of functions appearing as potentials of $\phi$.
This includes $f =\int_{\Omega} \nu(z)\phi(x-z) \diff z$ where $\Omega$ is a compact region with open interior; i.e for
Newton-like potentials (an  $L_2$ Sobolev theory for this case has been considered in \cite{HR-Extending}).
It also includes the case $f =\int_{\Omega} \nu(z)\phi(x-z) \diff \sigma( z)$ when $\Omega\subset \M$ is a
(Lipschitz region in) an embedded $d$-dimensional algebraic submanifold $\M$; this aligns with
single layer potential produced by $\phi$.

%
%
%
%
%
\subsection{Setup}

We assume $\Omega\subset \R^N$, which is $d\le N$ dimensional in the following sense.

%
%
%
\begin{assumption}
\label{ball_comparison}
We assume that $\Omega\subset \R^N$ is compact,  equipped with a finite, positive Borel
measure $\sigma:\mathcal{B}_{\Omega}\to [0,\infty)$,
and
there is 
a constant 
$\omega_{\Omega}$ so that
for any $x\in\R^N$ and any 
$r>0$,
$$\sigma\bigl(B_{\Omega}(x,r) \bigr)\le \omega_{\Omega} r^d$$
holds, where
 $B_{\Omega}(x,r) :=\{z\in \Omega\mid |x-z|<r\}$.
 \end{assumption}
 This holds (with $d=N$)  if $\sigma$ is Lebesgue measure on $\R^N$.
 Similarly, if $\M\subset \R$ is an embedded $d$-dimensional submanifold, $\Omega\subset \M$ is compact,
 and if $\sigma$ is the volume generated by the induced Riemannian metric, then $\sigma$ satisfies Assumption \ref{ball_comparison}.
 
 Assumption \ref{ball_comparison} is weaker than  $\sigma$ being a {\em $d$-measure} and $\Omega$ being a 
 {\em $d$-set}, 
 as considered in \cite{JW}.
 We point out some connections with this notion.
 Following \cite[Proposition 1.1]{JW}, it is clear that if $\Omega\subset \R^N$ has Hausdorff dimension $d$,
and $\sigma$ satisfies Assumption \ref{ball_comparison}, then $\sigma$ is absolutely continuous with respect to 
$d$-dimensional Hausdorff measure $H_d$. 
Conversely,  
if $\Omega\subset \R^N$ is a $d$-set 
and if $f\in L_{\infty}(H_d)$ 
is non-negative, 
then 
$\sigma: 
E\mapsto \int_E f(x) \diff H_d(x)$ 
satisfies Assumption 
\ref{ball_comparison}.

\begin{assumption}
\label{A_LPR}
For $L\in\N$,
we assume there are  constants 
$K<\infty$ and  $h_0>0$ so that if 
$\Xi\subset \Omega$ satisfies 
$h= \max_{z\in \Omega} \min_{\xi\in \Xi} |\xi-z|<h_0$, then 
 there is a ``local polynomial reproduction'' $a\in C(\Xi\times \Omega)$ 
so that for any $z\in\M$, the following three properties hold:
\begin{itemize}
\item  $\mathrm{supp}(a(\cdot,z))\subset B(z,Kh)$,
\item $\sum |a(\xi,z)|\le 2$, 
\item $\sum a(\xi,z) p(\xi) =p(z)$ for all $p\in \mathcal{P}_{L-1}(\Omega)$. 
\end{itemize}
\end{assumption}

This holds if $\Omega\subset \R^N$ is a Lipschitz region in $\R^N$ by \cite{Wendland}.
 We have shown in \cite{HRW_LPR} that
the second assumption holds if $\Omega$ is the union of Lipschitz domains of $d$-dimensional algebraic manifolds.

\medskip

We can use the local polynomial reproduction to approximate  the push-forward of elements of $ L_p(\sigma)$ as follows:
identify $\nu\in L_p(\sigma)$ with the measure $\tilde{\nu}$ on $\R^N$  which acts on continuous functions by
$\int_{\R^N} f\diff \tilde{\nu}(x)= \int_{\Omega} f(x) \nu(x) \diff \sigma(x)$. The LPR approximant is the discrete measure
\begin{equation}
\label{push-forward-approximant}
\tilde{\nu}_{\Xi} = \sum_{\xi\in\Xi} \Bigl(\int_{\Omega} \nu(y)a(\xi,y) \diff \sigma(x) \Bigr)\delta_{\xi}.
\end{equation}
This is obviously a finite measure on $\R^N$ and is thus in $B_{p,\infty}^{-N/p'}(\R^N)$. 
For $s\le -N/p'$ we can approximate
$\tilde{\nu}$ by $\tilde{\nu}_{\Xi}$.
Specifically, we obtain the estimate
$$ \|\tilde{\nu}- \tilde{\nu}_{\Xi}\|_{B_{p,q}^s(\R^N)} \le C h^{ (d-N)/p' - s} \|\nu\|_{L_p(\sigma)}$$
for  $(s,q) \preceq (-N/p', \infty)$.
This is demonstrated in  appendix \ref{S_measures}.

\subsection{Approximation of potentials  by kernels }

We also assume that the target function (i.e., the object to be approximated)  can be written as a potential.
\begin{assumption}
\label{A_target}
For the (surface spline or Mat{\' e}rn  kernel)
 $\rbf$, and 
 the
 polynomial order
 $\ord \in \N$,
 we assume
 the function $f\in C(\R^N)$ can be expressed as follows:
 there exist $\nu\in L_1(\sigma)$ and $Q\in \mathcal{P}_{m-1}(\R^N)$
 so that  for every $x\in \R^N$,
$$f(x)=\int_{\Omega} \nu(z) \rbf(x-z)\diff\sigma(z) +Q(x),$$ 
holds and $\int_{\Omega} \nu(z) P(z)\diff \sigma(z)=0$ for all $P\in\mathcal{P}_{\ord-1}(\R^N)$.
\end{assumption}

A consequence of this assumption is that for $\tilde{\nu}$, the push-forward of $\nu\, \diff \sigma$,
we have
$$f = \rbf*\tilde{\nu} +Q.$$

Suppose $f$ satisfies 
Assumption \ref{A_target}, with $\ord \in\N$
and
$\Xi \subset \Omega$ is a finite set of points for which
the hypotheses of Assumption \ref{A_LPR} hold
with 
$L\ge \ord$ (for now; we will assume more about $L$ later).
We then define 
the approximant
$$T_{\Xi}f:=\sum_{\xi\in \Xi} \Bigl(\int_{\Omega} a(\xi,z) \nu(z)\diff\sigma(z)\Bigr) \rbf(\cdot-\xi) +Q.$$
We can express this function on $\R^N$ as a convolution, namely,
$$T_{\Xi}f = \rbf* \tilde{\nu}_{\Xi} +Q$$
where $\tilde{\nu}_{\Xi} $ is given in (\ref{push-forward-approximant}).

Because the density $\nu$ annihilates polynomials of degree $m-1$, the following holds.
\begin{lemma}
If Assumptions \ref{A_LPR} 
and 
\ref{A_target}
hold with $L\ge \ord$, then
$T_{\Xi}f\in V_{\Xi,\ord}(\rbf).$
\end{lemma}
\begin{proof}
Since $T_{\Xi} f = \sum_{\xi\in \Xi} A_{\xi} \rbf(\cdot -\xi) +Q$,
with $ A_{\xi}= \int_{\Omega}\nu(z)   a(\xi,z) \diff \sigma(z) $
we have
$$\sum_{\xi\in \Xi} A_{\xi} P(\xi)
= \int_{\Omega}\nu(z)  \sum_{\xi\in X} a(\xi,z) P(\xi) \diff \sigma(z) = \int_{\Omega}\nu(z)   P(z) \diff \sigma(z) = 0.$$
Since $(A_{\xi}) \in \Lambda_{\Xi,\ord}$, the result follows.
\end{proof}

\subsection{Main theorem}
 
Recall $k:=N-d$, where $d$ is the intrinsic dimension of $\Omega$, see Assumption \ref{ball_comparison}.

%
%
%
 \begin{theorem}
 \label{T_main}
 Suppose $\rbf$ is a Mat{\'e}rn or surface spline kernel 
 and
 $\Omega\subset \R^N$ satisfies Assumptions \ref{ball_comparison} and \ref{A_LPR} with 
 $L>\max(\theta+N,\ord-1,-(s+k/p'))$.
 Then there is a constant $C$ so that 
 for $p,q\in[1,\infty]$, and $s<\theta+N/p$, as well as for the extremal smoothness  
 $s=\theta+N/p$ and $q=\infty$, 
 such that if $f$ satisfies Assumption \ref{A_target} 
 with $\nu\in L_p(\sigma)$ then
 $$
 \|f-T_\Xi f\|_{B_{p,q}^{s}(\R^N)} \le C h^{\theta+d+k/p-s} \|\nu\|_{L_p(\sigma)}
 $$
 holds.
 \end{theorem}
 This theorem provides a rare instance of negative norm error estimates for kernel methods. 
Please note, that the condition $(s,q) \preceq (-N/p', \infty)$ allows for negative values for $s$. This means, that for large negative values of $s$, the order of the local polynomial reproduction 
needs to be suitably large; namely it must satisfy
the relation 
$L+k/p'>-s$.
Error estimates in negative norms have been proven useful in several applications, see \cite[Chapter 5]{Thomee2006} or \cite{monsur:etal:2024}.

 \begin{proof}
 We can express the error between $f=  \rbf*\tilde{\nu}+Q$
 and 
 $T_{\Xi} f = \rbf*\tilde{\nu}_{\Xi} +Q$
 as 
 $$f-T_{\Xi} f =  \rbf*(\tilde{\nu}-\tilde{\nu}_{\Xi} ).$$
 By the mapping properties of convolution by $\rbf$ given in 
(\ref{Matern_Isomorphism}) (for Mat{\'e}rn kernels) and Lemma \ref{TPS_mapping}
(for surface splines), we have
$$ \|f-T_\Xi f\|_{B_{p,q}^{s}(\R^N)} \le C\|\tilde{\nu}-\tilde{\nu}_{\Xi} \|_{B_{p,q}^{s-(\theta+N)}(\R^N)}.$$
The theorem now follows from 
Lemma \ref{L_measure_approx}, which ensures that
$$\|\tilde{\nu}-\tilde{\nu}_{\Xi} \|_{B_{p,q}^{s-(\theta+N)}(\R^N)}
\le 
C h^{ (d-N)/p' - [s-(\theta+N)]} \|\nu\|_{L_p(\sigma)},$$
and from the observation $(d-N)/p' +N = d/p' +N/p = (d -d/p) +N/p =d + k/p$.
 \end{proof}

%
%
%
\section{Kernel approximation  on embedded manifolds}
\label{S_Main}
We  consider a setup similar to the one
in \cite{FW}. Let $\D, {\d},k\in \N$ with $\D={\d}+k$
and suppose 
 $\M\subset \R^{\D}$ is a  compact,\footnote{The requirement of compactness of $\M$ is one of convenience; it is 
 used to obtain a general trace result, but it can likely be replaced by a suitable condition of bounded geometry.}
  smooth embedded submanifold of dimension ${\d}$, possibly with boundary 
 and $\Omega \subset \M$ is a compact subset  with Lipschitz boundary
for which Assumption \ref{A_LPR} holds.
This allows $\M$ to be a compact subset of an algebraic manifold, as considered in \cite{HRW_LPR}. 

The induced Riemannian metric $g$, 
gives rise to distance $\dist_{\M}:\M\times \M\to \R$ and Borel measure $\sigma$. 
The  distance is equivalent to the Euclidean distance in the sense that
 there is a constant $\beta$,
 depending on $\Omega$, 
so that for all $z_1,z_2\in \Omega$,   the metric equivalence 
\begin{equation}
\label{Met_equiv}
  |z_1-z_2| \le  \dist_{\M}(z_1,z_2)\le \beta |z_1-z_2|
\end{equation}
holds. 
There is  a positive constant $\omega_{\Omega}$ so that
for $x\in\Omega$,
$
\sigma(\{z\in\Omega \mid |x-z|<r\})\le \omega_\Omega r^d.
$
Thus, Assumption \ref{ball_comparison} holds automatically in this setting.

\subsubsection*{Smoothness spaces on manifolds}
We define the Sobolev and Besov spaces on embedded submanifolds as in \cite[section 2.3]{FW}. Our goal
in this section is to provide a general trace result. This, too follows  \cite[section 2.3]{FW}, with small modifications
to treat manifolds with boundary.

Let  $\mathcal{A} = \{(U_j,\varphi_j)\mid j\le J\}$ be an  atlas for $\M$.
Note that compactness of $\M$ ensures the atlas  $\mathcal{A}$ can be assumed finite.
 Since $\M$ may have boundary, we have that for each $j$, $\varphi_j(U_j) = V_j$, which is either a ball in $\R^d$ or a half-ball.

By straightening out $\M$, each chart can be extended to a chart $\tilde{\varphi}_j:\tilde{U}_j\to \tilde{V}_j\subset \R^N$,
so that $\tilde{\mathcal{A}} = \{(\tilde{U}_j,\tilde{\varphi}_j)\mid j\le J\}$ is an atlas for a tubular neighborhood $\mathsf{N}$ of $\M$ and 
$\varphi_j = \tilde{\varphi}_j|_{\M}$.

 Now let $\{\tilde{\chi}_j\}$ be a partition of unity   subordinate to
 $\{\tilde{U}_j\}$, and let $\chi_j = \tilde{\chi}_j|_{\M}$ be the associated partition of unity of $\M$, which is subordinate
 to $\{U_j\}$.

 For a function $F$ defined in the tubular neighborhood, we introduce  maps $\tilde{ \pi}_j:C^{\infty}(\mathsf{N})\to C_c^{\infty}(\R^N)$ by
$$\tilde{ \pi}_j F (z) = \begin{cases} (\tilde{\chi}_j F) (\tilde{\varphi}_j^{-1} (z)), &z \in \tilde{V}_j,\\
0 &\text{otherwise}.
\end{cases}
$$
 Likewise, for a function $f$ defined on $\M$,  define  projections 
 $\pi_j:C^{\infty}(\M) \to C_c^{\infty}(\R_{\circ}^d)$, where
 $\R_{\circ}^d = \R^d$ or $\R_+^d= \R^{d-1}\times[0,\infty)$,
 depending on whether $V_j$ is a ball or half-ball,
 via the formula
$$ \pi_j f (z) = \begin{cases} (\chi_j f) (\varphi_j^{-1} (z)), &z \in V_j,\\
0 &\text{otherwise}.
\end{cases}
$$
Using this construction, define
$\|f\|_{B_{p,q}^s(\M)} := \sum_{j} \|\pi_j f\|_{B_{p,q}^s(\R_{\circ}^d)}$,
and note that we have the intertwining identity:
$$ \pi_j( \mathrm{Tr}_{\M} f )= \mathrm{Tr}_{\R_{\circ}^d}( \tilde{\pi}_j f).$$
If $\R_{\circ}^d = \R_{+}^d$ the trace $\mathrm{Tr}_{\R_{+}^d}$ is simply $\mathrm{Tr}_{\R_{+}^d}= \mathrm{rest}_{\R_+^{d}}\circ \mathrm{Tr}_{\R^d}$, 
where $\mathrm{rest}_{\R_+^{d}}$ is the restriction to the half-space.

\begin{lemma}
\label{L_trace}
Suppose
$\M\subset \R^N$, where $\M$  is a smooth, $d$-dimensional, compact embedded submanifold, possibly with boundary.
Let $k =N-d$ denote the manifold's codimension.
For $p\in[1 , \infty]$ and $s>0$,  the trace operator $\mathrm{Tr}_{\M}$
is continuous from 
$B_{p,q}^{s+k/p}(\R^N)$
to $B_{p,q}^{s}(\M)$,
while  for  $p\in[1, \infty)$,  and $m\in\N$, $\mathrm{Tr}_{\M}$ is continuous from 
$B_{p,1}^{m+k/p}(\R^N)$
 to $W_{p}^m (\M)$.
\end{lemma}
\begin{proof}
We begin by recalling the flat case, for
$\R_{\circ}^d\subset \R^N$.
When $\R_{\circ}^d = \R^d$,
 the result
\cite[Theorem 6.6.1]{BL}
ensures that
$\mathrm{Tr}_{\R^{\ell}}:
B_{p,q}^{S}(\R^{\ell+1}) 
\to  B_{p,q}^{S-1/p}(\R^{\ell})
$
as long as $S-1/p>0$. Iterating this
$k$ times, shows that
$\mathrm{Tr}_{\R^{d}}:
B_{p,q}^{S}(\R^N) \to  
B_{p,q}^{S-k/p}(\R^{d})$ is bounded, provided $S-k/p>0$.

In case $\R_{\circ}^d = \R_+^d$,
recall that $\|f \|_{B_{p,q}^{S}(\R_+^{d})} = \inf\{\|u\|_{ B_{p,q}^{S}(\R^{d}}\mid u\in B_{p,q}^{S}(\R^{d},\  f= \mathrm{rest}_{\R_+^{d}}u\}$, so
 the restriction map
$\mathrm{rest}_{\R_+^{d}}: B_{p,q}^{S-k/p}(\R^{d})\to B_{p,q}^{S-k/p}(\R_+^{d})$ is bounded (with unit norm).
Thus the composition $\mathrm{Tr}_{\R_+^{d}}:
B_{p,q}^{S}(\R^N) \to  
B_{p,q}^{S-k/p}(\R_+^{d})$ is bounded as well.
 
A similar result follows by 
applying the exotic case 
$\mathrm{Tr}_{\R^{d}}:B_{p,1}^{1/p}(\R^{d+1}) \to  L_{p}(\R^{d})$ given in
\cite[Eqn. (3) p. 156]{BL}
to the previous trace result
to get  boundedness of $\mathrm{Tr}_{\R^{d}}:B_{p,1}^{k/p}(\R^N) \to  L_p(\R^{d})$.
This then extends to
 $\mathrm{Tr}_{\R^{d}}:B_{p,1}^{\ell+k/p}(\R^N)\to W_p^\ell(\R^{d})$, since 
 partial
 differentiation is bounded on Besov spaces; i.e.,
 $D^{\beta}:B_{p,1}^{\ell+k/p}(\R^N) \to B_{p,1}^{\ell-|\beta|+k/p}(\R^N)$. The case
 of half-spaces follows similarly.
 
 The  general case follows from the intertwining identity  
 $ \pi_j( \mathrm{Tr}_{\M} f )= \mathrm{Tr}_{\R_{\circ}^d}( \tilde{\pi}_j f)$,
 which, along with the definition gives 
 $$\|\mathrm{Tr}_{\M} f \|_{B_{p,q}^s (\M)}  = \sum_{j=1}^J \|\pi_j( \mathrm{Tr}_{\M} f )\|_{B_{p,q}^s (\R_{\circ}^d)} 
 = \sum_{j=1}^J \| \mathrm{Tr}_{\R_{\circ}^d}( \tilde{\pi}_j f)\|_{B_{p,q}^s (\R_{\circ}^d)}.
 $$
 At this point, we can use the above
 flat trace results to the right hand side, to obtain
$$
 \|\mathrm{Tr}_{\M} f \|_{B_{p,q}^s (\M)} 
  \le 
C_1
\sum_{j=1}^J \|  ( \tilde{\pi}_j f)\|_{B_{p,q}^{s+k/p} (\R^N)} 
 \le C_2\|f\|_{B_{p,q}^{s+k/p} (\R^N)}.
$$
The first constant $C_1=\| \mathrm{Tr}_{\R_{\circ}^d}\|_{B_{p,q}^{s+k/p} (\R^N)\to B_{p,q}^s (\R_{\circ}^d)}$
is the bound on the trace operator,
and the second constant incorporates $\max_{j\le J} \|\tilde{\pi}_j \|_{B_{p,q}^{s+k/p} (\R^N) \to B_{p,q}^{s+k/p} (\R^N)}$, which is 
 finite by compactness of $\M$ (yielding $J<\infty$) and  by the boundedness of $\tilde{\pi}_j $, which follows from the fact that the  operations of multiplication $f\mapsto \tilde{\chi}_j f$ and
 smooth change of variable $f\mapsto f\circ \tilde{\varphi}_{j}$ are both bounded on Besov spaces. A completely similar result holds in case of 
 trace into Sobolev space.
\end{proof}

\subsection{Approximation on  \texorpdfstring{$\M\subset \R^N$}{}}
For a finite set 
$\Xi\subset \M\subset \R^N$, we recall the definitions of $\Lambda_{\Xi,\ord}$ and $V_{\Xi,\ord}(\Phi)$ given in  section \ref{SS_PD_kernels}.
The following theorem results from the  estimates of Theorem \ref{T_main} and properties of the trace operator.
\begin{theorem}
\label{T_manifold_approximation}
Suppose
$\Omega\subset\M$ 
is a compact subset of $d$-dimensional submanifold $ \M\subset \R^N$, 
and that 
Assumption \ref{A_LPR}   holds on $\Omega$. 
Let $\rbf  =  \tphi$ or $\tps$ and $\Phi_{\theta}:\M\times \M \to \R:(x,y)\mapsto \rbf(x-y)$.

If  $P\in\mathcal{P}_{\ord-1}(\M)$ and
$\nu\in L_p(\M)$ satisfies $\int_{\Omega} \nu(z) Q(z) \diff \sigma(z)=0$ for all $Q\in \mathcal{P}_{\ord-1}(\M)$,
then
for $p\in[1,\infty]$,
 $F =  \int_{\Omega} \Phi_{\theta}(\cdot,z) \nu(z)\diff \sigma(z) +P$ is an element of $ B_{p,\infty}^{\theta+d}(\M)$.

For  $p,q\in[1,\infty]$ and $0 < s <\theta+d/p$, 
as well as for the extremal smoothness  
$s=\theta+d/p$ and $q=\infty$,  the following holds:
if $\Xi\subset \Omega$ is sufficiently dense then 
$$\dist_{B_{p,q}^s(\M)} \bigl(F,V_{\Xi,\ord}(\Phi_{\theta}) \bigr):= 
\inf_{U\in V_{\Xi,\ord}(\Phi_{\theta})} \|F- U\|_{B_{p,q}^s(\M)} \le 
C h^{\theta+d-s} \|\nu\|_{L_p(\M)}.$$
\end{theorem}
\begin{proof}
We  consider cases $p>1$ and $p=1$, and assume that $\nu\in L_p(\M)$ is supported in $\Omega$. 
Define 
$f:= 
\int_{\M} 
\rbf(\cdot-z) \nu(z)
\diff \sigma(z)+P
$, 
which is smooth away from 
$\Omega\subset \R^N$. 
Consider the  compactly 
supported distribution 
$\nu\otimes \delta_{\M}:
g
\mapsto 
\int_{\M}g(y)\nu(y)
\diff\sigma(y)$.

If $\nu\in L_1(\M)$, 
then $\nu\otimes \delta_{\M}$ 
is a finite Radon measure, so 
$\nu\otimes \delta_{\M}
\in 
B_{1,\infty}^0(\R^N)$; see \cite[Theorem. 2.2]{Kabanava_2008} or Remark \ref{R_delta}.
Thus $f=\rbf*(\nu\otimes \delta_{\M})\in B_{1,\infty}^{\theta+N}(U)$ 
for a neighborhood $U$ of $\M$
by (\ref{Matern_Isomorphism})  
in case of Mat{\'e}rn kernels or
(\ref{TPS_mapping}) 
for surface spline kernels. Applying the trace theorem gives 
$F =\mathrm{Tr}f\in B_{1,\infty}^{\theta +d}(\M)$.

If $\nu\in L_p(\M)$, $p\in (1,\infty]$, then  $\nu\otimes \delta_{\M}$ satisfies, for any $\phi\in \mathcal{S}(\R^N)$ 
$$
|\langle \nu\otimes \delta_{\M},\phi\rangle| \le \|\mathrm{Tr}\phi\|_{L_{p'}(\M)}\|\nu\|_{L_p(\M)} \le
C\|\phi\|_{B_{p,1}^{k/p'}(\R^N)}\|\nu\|_{L_p(\M)}
$$
by Lemma \ref{L_trace},
so $\nu\otimes \delta_{\M}\in (B_{p',1}^{k/p'})'(\R^N) $ which equals $ B_{p,\infty}^{-k/p'}(\R^N) $ by \cite[Corollary 6.2.8]{BL}.
By 
(\ref{Matern_Isomorphism}) or
(\ref{TPS_mapping}), 
it follows that $f\in B_{p,\infty}^{\theta+N-k/p'}(U) = B_{p,\infty}^{\theta+d/p' +N/p}(U) $
for a neighborhood of $\M$.
Using 
$d(1-1/p)+N/p=d+(N-d)/p$, 
the trace theorem \ref{L_trace} 
gives $F =\mathrm{Tr}f\in B_{p,\infty}^{\theta +d}(\M)$.

Using the approximation scheme of section \ref{S_potentials}, set  
$u= T_\Xi f=\sum_{\xi\in \Xi} A_{\xi} \phi(\cdot-\xi)+P$
and  
$U= \mathrm{Tr}_{\M} u$.
In particular, $U = P+\sum_{\xi\in \Xi} A_{\xi} \Phi_{\theta}(\cdot,\xi) $. By the
hypothesis that $ \nu\perp \mathcal{P}_{\ord-1}(\M)$, 
we have that 
$$\sum_{\xi\in \Xi} A_{\xi} Q(\xi) = \sum_{\xi\in \Xi} \int_{\Omega} a(\xi,z) Q(\xi)\nu(z)  \diff z= \int_{\Omega} Q(z) \nu(z) \diff z  = 0,$$
so $U\in V_{\Xi,\ord}(\Phi_{\theta})$.
By Lemma \ref{L_trace} followed by   Theorem \ref{T_main}
we have
$$\|F-U\|_{B_{p,q}^{s}(\M)}\le C \|f-u\|_{B_{p,q}^{s+k/p}(\R^N)} \le  C h^{\theta +d+k/p -(s+k/p) }\|\nu\|_{L_p(\M)}$$
and the result follows.
\end{proof}

\subsection{Range of the single layer potential operator}
\label{SS_SLP}
We note that in some cases, the class of  target functions 
\begin{equation}\label{eq:Ap}
\mathcal{A}_p := 
\Bigl\{
F = 
\int_{\M} \kernel(\cdot, z) \nu(z) \diff \sigma(z) 
\mid 
\nu\in L_p(\Omega),\ 
\nu \perp \mathcal{P}_{\ord-1}(\M)
\Bigr\}
+
 \mathcal{P}_{\ord-1}(\M)
\end{equation}
is a standard smoothness space. We now discuss 
the mapping properties of the 
operator 
$$\mathcal{V}:\nu\mapsto \int_{\M} \nu(z)\Phi(\cdot,z)\diff z$$
and consider
two  cases where the range of $\mathcal{V}$ 
is well understood.

\medskip

\begin{remark} 
\label{R_Spheres} 
If $\M$ is a sphere, $\M=\Sph^{\d}$, then  $\mathrm{Range}(\mathcal{V})$ is easily determined.

In this case the  polynomial $P\in \mathcal{P}_{\ord-1}(\Sph^d)$ has a spherical harmonic expansion
$$P=\sum_{\ell\le \ord-1}\sum_{\mu\le N_{d,\ell}} \langle P, Y_{\ell}^\mu \rangle  Y_{\ell}^\mu.$$
Furthermore, the kernel satisfies the expression
$\kernel(x,y) = \rbf(x-y) = \varphi(x\cdot y)$  and can be expanded 
with the Mercer-like series
$$\kernel(x,y) = \sum_{\ell=0}^{\infty} \sum_{\mu\le N_{d,\ell}} c_{\ell,\mu} Y_{\ell}^\mu(x)Y_{\ell}^\mu(y)$$
with expansion coefficients  $c_{\ell,\mu}>0$ for all $\ell\ge \ord-1$.
For many such kernels (including restricted surface splines and Mat{\'e}rn kernels) 
the coefficients have been  determined analytically in \cite{Baxter} and \cite{Odell}.

Thus, the operator 
$\mathcal{V}$ 
can be expressed on 
$
\mathcal{P}^{\perp}_{\ord-1}
(\Sph^d)
$ 
as
$$
\mathcal{V}\nu
=
\sum_{\ell,m} 
c_{\ell,\mu} 
\langle 
\nu, Y_{\ell}^\mu
\rangle Y_{\ell}^\mu.
$$
It follows that
$
\nu
\mapsto 
\mathrm{Proj}_{\ord-1} \nu
+ 
\mathcal{V}
(1-\mathrm{Proj}_{\ord-1})\nu
$
is an invertible map from 
$L_2(\Sph^d)$ to 
$W_2^{\theta+d}(\Sph^d)$.
So  if 
$f\in W_2^{\theta+d}(\Sph^d)$, 
then 
$f=
\int_{\Sph^{\d}} \nu(y)\kernel(\cdot,y)
\diff\mu(y)+P
$
with 
$P 
=
\mathrm{Proj}_{\ord-1} f$, $\nu \in L_2(\Sph^d)$ 
and equivalent  norms 
$
\|\nu\|_{L_2(\Sph^{\d})} 
\sim 
\|f\|_{W_2^{\theta+d}
(\Sph^{\d})/
\mathcal{P}_{\ord-1}(\Sph^d)} 
$.
\end{remark}
\medskip

\subsubsection*{Smooth boundaries}
In this case we consider $\M=\partial \varUpsilon$ for a compact region  $\varUpsilon\subset \R^{\D}$ 
with smooth boundary
(so $\d=\D-1$ and $k=1$),
and the measure $\sigma$ is the surface measure on $\partial \varUpsilon$.
We restrict attention to the case that $\phi$ is the fundamental
solution to  an elliptic, homogeneous constant coefficient differential operator of order $2\DiffOrd$. 
This means $\phi=\tphi$ or $\tps$ with $\theta+N=2\DiffOrd$.

In this case, 
$\tilde{\mathcal{V}}:
\nu\mapsto \int_{\M} \nu(z) \rbf(x-z)\diff \sigma(z)$ is a {\em single layer potential}.
If $p\in(1,\infty)$ 
and $s\in\R$
then the composition
$$\mathcal{V} := 
\mathrm{Tr} \circ \tilde{\mathcal{V}}$$
 of 
$\mathcal{V}$ with  the trace to $\partial \varUpsilon$
extends to a bounded map from 
$W_p^{s}(\partial  \varUpsilon)$  to  
$W_p^{s+2\DiffOrd-1}(\partial  \varUpsilon)$
by \cite[Theorem 3.2]{Duduchava} (see also \cite[Lemma 3.7]{HLayer}).  
In particular, when $s=0$,
$\mathcal{V}$ is bounded from $L_p(\partial\varUpsilon)$ to 
$W_p^{\theta+d}(\partial  \varUpsilon)$,
since $2\DiffOrd-1 =\theta+N-1=\theta+d$.
More can be said in this case, because $\mathcal{V}$ is a pseudodifferential operator. 

The next lemma shows that 
$\mathcal{A}_p = 
W_p^{\theta+d}(\partial \varUpsilon)/ \mathcal{P}_{\ord-1}(\Sph^d)
=W_p^{2\DiffOrd-1}(\partial  \varUpsilon)/\mathcal{P}_{\ord-1}(\Sph^d)$ when $1<p<\infty$.
To treat the conditionally positive definite case, we augment $\mathcal{V}$ by a polynomial as follows:
let 
$(p_j)_{j\le L_{\ord-1}}$ be a basis for  $\mathcal{P}_{\ord-1}(\M)$. Then $\mathcal{V}^{\sharp} $
is 
defined by 
$$\mathcal{V}^{\sharp} \begin{pmatrix} \nu\\ \vec{c}\end{pmatrix} 
=
\begin{pmatrix} \mathcal{V} & P\\ P^T& 0\end{pmatrix}  \begin{pmatrix} \nu\\ \vec{c}\end{pmatrix} 
= 
\begin{pmatrix} f\\ \vec{d}\end{pmatrix}
$$ where 
$
f
= 
\mathcal{V}\nu +\sum_{j=1}^{L_{\ord-1}} c_j p_j
$
and $d_j = \langle \nu,p_j\rangle$.

\begin{lemma} \label{lem:Apboundary} Suppose $\Upsilon\subset \R^N$ is a compact region
with smooth boundary, 
and let $\M:=\partial \Upsilon$.
 Then for an integer $\DiffOrd>N/2$,
 $\theta=2\DiffOrd-N$ and $\ord$ such that
$\ord\ge 0$ if $\rbf=\tphi$ and 
let 
$\ord 
\ge  
\lfloor\frac{\theta}{2}\rfloor+1
$
if $\rbf=\tps$, the following holds:
for $p\in(1,\infty)$ and $s\in\R$, the map
$$
\mathcal{V}^{\sharp}
:
W_p^{s}(\M)\times \R^{L_{\ord-1}} 
\to 
W_p^{s+2\DiffOrd-1}(\M)\times \R^{L_{\ord-1}}
$$
is  boundedly invertible. 
In particular, 
$\mathcal{A}_p 
= 
W_p^{2\DiffOrd-1}(\M)/\mathcal{P}_{\ord-1}(\M)
$.
\end{lemma}
\begin{proof} 
One can calculate, 
as in 
\cite[Section 5.3]{HLayer},
 the principal symbol  
 $\sigma(\mathcal{V})$ of the 
 pseudodifferential operator 
 $\mathcal{V}$.
 In particular,
 the formula for $k=j=0$ in 
 \cite[Lemma 5.8]{HLayer}
shows that $\mathcal{V}$
is elliptic and 
polyhomogeneous of order 
$2k_0-1$.
 The standard theory of pseudodifferential operators permits us to show that $\mathcal{V}$ has a right parametrix 
$R$ which has order $1-2k_0$; i.e., $\mathcal{V}R = I+K$
is a perturbation of the identity $I$ by a
smoothing operator 
$K:
\mathcal{D}'(\M)
\to 
\mathcal{D}(\M)
$
(see \cite[Theorem 8.6]{Grubb}).
Consequently, 
$R:
W_p^{s+2\DiffOrd-1}(\M)
\to 
W_p^{s}(\M )
$ 
and
$K:
W_p^s(\M )\to 
W_p^s(\M)
$
is compact.
From this, it follows that 
the range of $\mathcal{V}R$ 
in $W_p^s(\M)$ has finite codimension,
and so 
the range of $\mathcal{V}$ is has 
finite codimension as well. It is therefore  closed.
Finally, because $L_{\ord-1} =\dim\mathcal{P}_{\ord-1}(\M)<\infty$, 
the range of
$\mathcal{V}^{\sharp}$ is closed in
$W_p^{s+2\DiffOrd-1}(\M)\times \R^{L_{\ord-1}}$.

 To show that $\mathcal{V}^{\sharp}$ is injective, 
 assume that $\mathcal{V}^{\sharp}(\nu,\vec{c})^T = 0$, which  ensures that $\langle \nu, p_j\rangle =0$ for all $j$, 
  and that 
  $\mathcal{V}\nu
  \in 
  \mathcal{P}_{\ord-1}(\M)
  \subset 
  C^{\infty}(\M)$. 
  By ellipticity of $\mathcal{V}$, it follows that 
  $\nu\in \mathcal{P}^{\perp}_{\ord-1}
  \cap C^{\infty}(\M)$. 
  In particular, $\nu$ is a regular distribution.
 It follows that
  $\nu\otimes \delta_{\partial \varUpsilon}$ is a Borel measure, and therefore a functional on $\mathcal{N}(\phi)$.
  Its Riesz representer is (up to a constant) $(\nu\otimes \delta_{\partial \varUpsilon}) *\phi$, which has 
  native space semi-norm  
  $$\|(\nu\otimes \delta_{\partial \varUpsilon}) *\phi\|_{\mathcal{N}(\phi)}^2 =
  \int_{\partial\varUpsilon} (\mathcal{V}\nu)(x)  \nu(x)\diff \sigma(x)  =  
  \int_{\partial\varUpsilon}  \bigl(\mathcal{V}\nu(x) + \sum c_j p_j(x)\bigr) \nu(x)\diff \sigma(x) = 0.$$
  The middle equation uses the fact that  
  $\nu\in \mathcal{P}^{\perp}_{\ord-1}
  $.
  This ensures
  that $\nu\otimes \delta_{\partial \varUpsilon}$ vanishes on $\mathcal{N}(\phi)$, which implies that $\nu=0$, 
  since $C_c^{\infty}(\R^N)\subset \mathcal{N}(\phi)$. 
  By linear independence of $(p_j)$, we have $\vec{c}=0$.

  By symmetry of the kernel,
   the adjoint map of $\mathcal{V}$ is
$(\mathcal{V})^*:
W_{p'}^{-(s+2\DiffOrd-1)}(\partial \Upsilon)
\to 
W_{p'}^{-s}(\partial \Upsilon)$, 
where 
$\mathcal{V}^* T(x) 
= \langle T, \kernel(x,\cdot)\rangle$.
Similarly,
the adjoint map of $\mathcal{V}^{\sharp}$ is 
$$
(\mathcal{V}^{\sharp})^*
:
W_{p'}^{-(s+2\DiffOrd-1)}(\M)\times \R^{L_{\ord-1}}
\to 
W_{p'}^{-s}(\M)\times \R^{L_{\ord-1}}.$$
Since this uses the same kernel 
(only the domain and codomain have changed), 
$(\mathcal{V}^{\sharp})^*$
is   injective, and
the fact that the range of $\mathcal{V}^{\sharp}$  is closed ensures that
 $$
 \mathrm{Range}(\mathcal{V}^{\sharp})
 = 
 \bigl(\mathrm{Null}(\mathcal{V}^{\sharp})^*\bigr)^{\perp} 
 =(\{0\})^{\perp}
 = W_p^{s+2\DiffOrd-1}(\M)\times \R^{L_{\ord-1}}.
 $$ 
 Thus
  $\mathcal{V}^{\sharp}:W_p^{s}(\M)\times \R^{L_{\ord-1}}\to W_p^{s+2\DiffOrd-1}(\M) \times \R^{L_{\ord-1}}$ 
  is boundedly invertible.
\end{proof}

\subsection{A Bernstein inequality}
Suppose $\M\subset \R^N$ is compact.
The restricted kernel $\kernel:\M\times \M\to \R:(x,y)\mapsto \rbf(x-y)$
for  $\rbf=\tphi$ or $\tps$
has the following two properties:
\begin{itemize}
\item $\mathcal{N}_m (\kernel)\sim W_2^{(\theta+d)/2}(\M)$,
\item $\kernel(\cdot,z) \in B_{2,\infty}^{\theta+d/2}(\M)$.
\end{itemize}
Note, in particular, that     there
is a bounded extension operator $E:W_2^{(\theta+d)/2}(\M)\to W_2^{(\theta+N)/2}(\R^N)$ which
guarantees, by the continuous embedding $W_2^{(\theta+N)/2}(\R^N)\subset \mathcal{N}(\rbf)$, that
$$
|u|_{\mathcal{N}(\kernel)}
\le 
|Eu|_{\mathcal{N}(\rbf)}
\le
\|Eu \|_{W_2^{(\theta+N)/2}(\R^N)}
\le C\|u\|_{W_2^{(\theta+d)/2}(\M)}.
$$

The following Bernstein inequality treats functions in  $V_{\Xi,m}(\kernel)$
by relating their high smoothness norm (smoothness between $\theta/2+d/2$ and $\theta+d/2$) to the 
native space norm. In this subsection and the next we
restrict focus to smoothness measured in $L_2$, where $W_2^t(\M) = B_{2,2}^t(\M)$.
\begin{theorem} 
\label{T_Bernstein}
If $0\le\nu \le  s<\theta/2$
there is a $C>0$ so that the following holds:
for  $\Xi\subset \M$
and
$U = \sum_{x_j\in\Xi} a_j \Phi_{\theta}(\cdot,x_j)+P \in V_{\Xi,m}(\Phi_{\theta})$,
%
$$
 \|U\|_{W_2^{(\theta+d)/2+s}(\M)}
 \le  
C\begin{cases}  \|U\|_{L_2(\M)}+ q^{-s} |U|_{\mathcal{N}_{\ord}(\kernel)}&m\ge 1,\\ 
 q^{-s} \|U\|_{\mathcal{N}(\kernel)}&m=0.
 \end{cases}
$$
%
 \end{theorem}

 To treat the CPD case, we use the following lemma, which permits us
 to treat the quotient spaces of the form $W_2^{t}/\mathcal{P}_{m-1}$.
  We note that this quotient space was also used in the original version of the Bramble-Hilbert theorem, see \cite{Bramble:Hilbert:1970}.
%
%
%
\begin{lemma}\label{L_equiv}
 If $Y$ and $X$ are Hilbert spaces with continuous embedding $Y\subset X$ and if $\varPi\subset Y$ is finite dimensional, let $P:X\to \varPi$
 be the $X$ orthogonal projector onto $\varPi$. Then  
 $$\langle y,z\rangle_\circ :Y\times Y\to \R:(y,z)\mapsto \langle y,z\rangle_X + \langle y-Py, z-Pz\rangle_Y$$
 induces a norm $\|y\|_\circ$ which satisfies, for constants $0<c_\circ\le C_{\circ}<\infty$ the estimate
  $$c_{\circ} \|y\|_Y\le \|y\|_\circ\le C_{\circ} \|y\|_Y,$$
 and $P:Y\to \varPi$ is orthogonal with respect to this inner product.
 \end{lemma}
%
%
%
 \begin{proof} Assume without loss that $\|y\|_X\le \|y\|_Y$.
 There exists a  constant $c_B$ so that $\|p\|_Y\le c_B\|p\|_X$ for all $p\in \varPi$. Then $P$ is bounded on $Y$,
 with $\|P y\|_Y \le c_B \| y\|_Y$. It follows that  $\langle y,z\rangle_{\circ}$ is an inner product on $Y$.
Because $\|y\|_\circ \le \sqrt{2+c_B}\|y\|_Y$, the open mapping theorem guarantees $\|y\|_Y\le \Gamma \|y\|_\circ$.
 \end{proof}
 %
 %
 %
 %
 %
 \begin{proof}[Proof of Theorem \ref{T_Bernstein}]
 For $\ord=0$, we have $u=\sum_{x_j\in\Xi} a_j \rbf(\cdot-x_j) \in V_X(\rbf)$ and $U=\mathrm{Tr}\,u$. In that case,
$$\|U\|_{W_2^{(\theta+d)/2+s}(\M)} \lesssim   \|u\|_{W_2^{(\theta+N)/2+s}(\R^N)}\lesssim  q^{-s} \|u\|_{\mathcal{N}(\Phi_{\theta})} 
=  q^{-s}\|U\|_{\mathcal{N}(\rbf)}$$
follows where the first inequality is the trace theorem, the second inequality is the Euclidean Bernstein estimate
\cite[Theorem 3.1]{HR-Extending}, and the final equality follows because the two native space norms, squared, equal
 $\sum a_j a_k \rbf(x_j-x_k) = \sum a_j a_k \kernel(x_j,x_k)$.
 
 Let $u=\sum_{x_j\in\Xi} a_j \rbf(\cdot-x_j) +P\in V_{X,m}(\rbf)$.
 It follows that 
$U=\mathrm{Tr}_{\M}\,u.$
Thus 
$$\|U \|_{W_2^{(\theta+d)/2+s}(\M)}\le C_1 \|u\|_{W_2^{(\theta+N)/2+s}(B)},$$
where $B\subset \R^N$ is a ball sufficiently large that $\M\subset B$. 
Let $Q\in \mathcal{P}_{m-1}(\R^N)$ denote the 
average Taylor polynomial of $u$ over $B$;  then by
 Bramble-Hilbert,
$\|u-Q\|_{W_2^{(\theta+N)/2+s}(B)}\le C_2 |u|_{W_2^{(\theta+N)/2+s}(B)}$.
Applying the trace estimate followed by 
the 
 Euclidean Bernstein estimate \cite[Theorem 3.1]{HR-Extending},
 gives, for $C_3=C_1C_2$ and an enlarged constant $C_4$, the estimate
$$
\|U-Q \|_{W_2^{(\theta+d)/2+s}(\M)}\le 
 C_3 |u|_{W_2^{(\theta+N)/2+s}(\R^N)}
\le  C_4 q^{-s} |u|_{\mathcal{N}_{\ord}(\rbf)}.
$$
Since $ \sum a_j a_k \rbf(x_j - x_k)=   \sum a_j a_k \kernel(x_j,x_k)$,
it follows that $\|U-Q \|_{W_2^{(\theta+d)/2+s}(\M)}\le C_4 q^{-s} |U|_{\mathcal{N}_{\ord}(\kernel)}$.

Apply Lemma \ref{L_equiv} with $X=L_2(\M)$, $Y=W_2^{(\theta+d)/2+s}$
and $\|f\|_{\circ}^2 := \|f\|_{L_2(\M)}^2 + \|f-Pf\|_{W_2^{(\theta+d)/2+s}(\M)}^2$.
It then follows that 
$ \|U-P U\|_{\circ} \le \|U-Q\|_{\circ}$,
since $P$ is the orthogonal projector.
By metric equivalence,
$\|U-PU\|_{W_2^{(\theta+d)/2+s}(\M)} \le C_4 \|U-Q\|_{W_2^{(\theta+d)/2+s}(\M)}$
holds, and so
$$
\|U\|_{W_2^{(\theta+d)/2+s}(\M)}
\le C_4( \|PU\|_{L_2(\M)} + \|U-PU\|_{W_2^{(\theta+d)/2+s}(\M)})
\le C_5( \|U\|_{L_2(\M)} +  q^{-s} |U|_{\mathcal{N}_{\ord}(\kernel)})
$$
holds and the theorem follows.
\end{proof}

%
%
\subsection{Interpolation error}
The Bernstein inequality allows novel interpolation error estimates for error in Sobolev norms $W_2^{t}(\M)$ with
$t>(\theta+d)/2$.
We begin by recalling the standard doubling result:
%
%
\begin{proposition}
\label{P_doubling}
For  $Q\in\mathcal{P}_{m-1}(\R^N)$ and $\nu\in L_2(\M)$  with $\nu\perp \mathcal{P}_{m-1}(\M)$,
let $f=\int_{\M} \nu(y)\rbf(\cdot-y)\diff y+Q$. For $t<(\theta+d)/2$ and a sufficiently dense $\Xi\subset\M$ (i.e., $h=\max_{x\in\M} \dist(x,\Xi)$ is small enough), we have
\begin{eqnarray*}
\|f-I_{\Xi} f\|_{W_2^{t}(\M)}&\le& C h^{\theta+d-t} \|\nu\|_{L_2(\M)}\\
|f-I_{\Xi} f|_{\mathcal{N}_m(\rbf)}&\le& C h^{(\theta+d)/2} \|\nu\|_{L_2(\M)}
\end{eqnarray*}
\end{proposition}
%
\begin{proof}
Because the proof for the Mat{\'e}rn case $\rbf=\tphi$ is substantially easier, we assume $\rbf=\tps$.

By the zeros estimate (\cite[Lemma 10]{FW} \cite[Theorem A.11]{HNW-poly}) we have
$$\|f-I_\Xi f\|_{W_2^{t}(\M)} \le Ch^{(\theta+d)/2-t} \|f-I_\Xi f\|_{W_2^{(\theta+d)/2}(\M)}.$$
Let $B \subset \R^{N}$ contain $\M$ and let $Q \in\mathcal{P}_{m-1}(\R^N)$ be the orthogonal projection with
respect to $W_2^{(\theta+N)/2}(B)$.
By applying Lemma \ref{L_equiv}  with $X=L_2(\M)$ and $Y= W_2^{(\theta+d)/2}(\M)$ 
and Bramble-Hilbert, as in the proof of
Theorem \ref{T_Bernstein}, we have
$$\|f-I_\Xi f\|_{W_2^{(\theta+d)/2}(\M)}\le C\left(\|f-I_{\Xi} f\|_{L_2(\M)} +|f-I_{\Xi} f|_{W_2^{(\theta+N)/2}(\R^N)}\right).$$
By rearranging terms, using $h^{(\theta+d)/2 -t} \|f-I_{\Xi}f\|_{L_2(\M)} \le \frac12 \|f-I_{\Xi}f\|_{W_2^t(\M)}$, we have
$$\|f-I_{\Xi} f\|_{W_2^t(\M)} 
\le 
Ch^{(\theta+d)/2-t} |f-I_{\Xi} f|_{\mathcal{N}_m(\rbf)}.$$
Note that for a finite set $Z\subset \R^N$
and any suitable finite measure
$\mu= \sum_{\zeta \in Z}  a_z \delta_{z}$ the semi-inner product $\langle f, \rbf*\mu\rangle_{\mathcal{N}_m(\rbf)} = \sum_{z\in Z} a_z\int_{\M} \nu(y)\rbf(z-y)\diff \sigma(y) =\int_{\M}  \bigl(\rbf*\mu\bigr)(y)\nu(y)\diff \sigma(y)$.
By continuity, for any $g\in \mathcal{N}_m(\rbf)$, $\langle f,g\rangle_{\mathcal{N}(\rbf)} = \int_{\M} \nu(y)g(y)\diff \sigma(y)$.
Thus $|f-I_{\Xi} f|_{\mathcal{N}_m(\rbf)}^2 \le \|\nu\|_{L_2(\M)} \|f-I_{\Xi}f\|_{L_2(\M)}$.
This implies $\|f-I_{\Xi} f\|_{W_2^{t}(\M)}
\le  C\|\nu\|_{L_2(\M)}^{1/2}  h^{(\theta+d)/2-t}  \|f-I_{\Xi}f\|_{L_2(\M)}^{1/2}$
and  
$\|f-I_{\Xi} f\|_{L_2(\M)}
\le h^{\theta+d} \|\nu\|_{L_2(\M)}$. 
From this estimate, the proposition follows.
\end{proof}
We point out that Proposition \ref{P_doubling} with $t=0$ provides, for a quasi-uniform point set $\Xi \subset \M$ with $h\sim |\Xi|^{-\frac{1}{d}}$, the optimal approximation rate for all even non-linear approximation methods using function evaluations on $\Xi$, see \cite{Krieg:Sonnleitner:2025}.

\begin{theorem}\label{main}
Suppose
$\Omega\subset\M$ 
is a compact subset of $d$-dimensional submanifold $ \M\subset \R^N$, 
and that 
Assumption \ref{A_LPR}   holds on $\Omega$. 
Let $\rbf  =  \tphi$ or $\tps$ so that $\rbf$ is CPD with respect to $\ord \in\N_0$.
For a target function 
 $f =  \int_{\Omega} \Phi_{\theta}(\cdot,z) \nu(z)\diff \sigma(z) +P$ 
 where
$\Phi_{\theta}:\M\times \M \to \R:(x,y)\mapsto \rbf(x-y)$, 
$\nu\in L_2(\M)$ satisfies $\int_{\Omega} \nu(z) Q(z) \diff \sigma(z)=0$ for all $Q\in \mathcal{P}_{\ord-1}(\M)$,
and    $P\in\mathcal{P}_{\ord-1}(\M)$, if
 $0\le s < \theta/2$, 
 the  interpolation estimate
 $$\|f- I_{\Xi} f\|_{W_2^{(\theta+d)/2+s}(\M)}\le C (\rho_{\Xi})^s h^{(\theta+d)/2-s} \|\nu\|_{L_2(\M)}$$
 holds. Here $\rho_{\Xi} := h/q$ is the
{\em mesh ratio} of $\Xi$.
 \end{theorem}
\begin{proof}
Under conditions of Theorem \ref{T_manifold_approximation}, 
the approximant
$U_f = T_{\Xi} f \in V_{\Xi,\ord}(\Phi_{\theta})$  satisfies
\begin{eqnarray}
\|f-U_f\|_{W_2^{(\theta+d)/2+s}(\M)} 
&\le& C h^{(\theta+d)/2-s}\|\nu\|_{L_2(\M)},\label{strong}\\
\|f-U_f\|_{L_2(\M)} 
&\le& C h^{\theta+d}\|\nu\|_{L_2(\M)}.\label{weak}
\end{eqnarray}
It follows that
\begin{eqnarray*}
\|f-I_{\Xi} f\|_{W_2^{(\theta+d)/2+s}(\M)} &\le & \|f-U_f\|_{W_2^{(\theta+d)/2+s}(\M)} +  \|U_f-I_{\Xi} f\|_{W_2^{(\theta+d)/2+s}(\M)} \\
&\le& 
C \left(h^{(\theta+d)/2-s} \|\nu\|_{L_2(\M)} + \|U_f-I_{\Xi}  f\|_{L_2(\M)}+ q^{-s} |U_f-I_{\Xi} f|_{\mathcal{N}(\kernel)}\right).
\end{eqnarray*}
The second inequality follows by applying  Theorem \ref{T_Bernstein}  to $U_f-I_{\Xi}f\in V_{\Xi,\ord}(\kernel)$.
The theorem follows by applying
 the triangle inequality
$$\| U_f-I_{\Xi} f\|_{X}\le \|f-U_f\|_{X} +\|f-I_{\Xi} f\|_{X}$$
(with $X=L_2(\M)$ or $\mathcal{N}(\kernel)$) followed by
inequalities (\ref{strong}) and (\ref{weak}) and  Proposition \ref{P_doubling}.
\end{proof}
 
%
%
%
\subsection{Numerical results on spheres\label{sec:numerics_higher_norms}}
\begin{figure}[htb]
\centering
\includegraphics[width=0.5\textwidth]{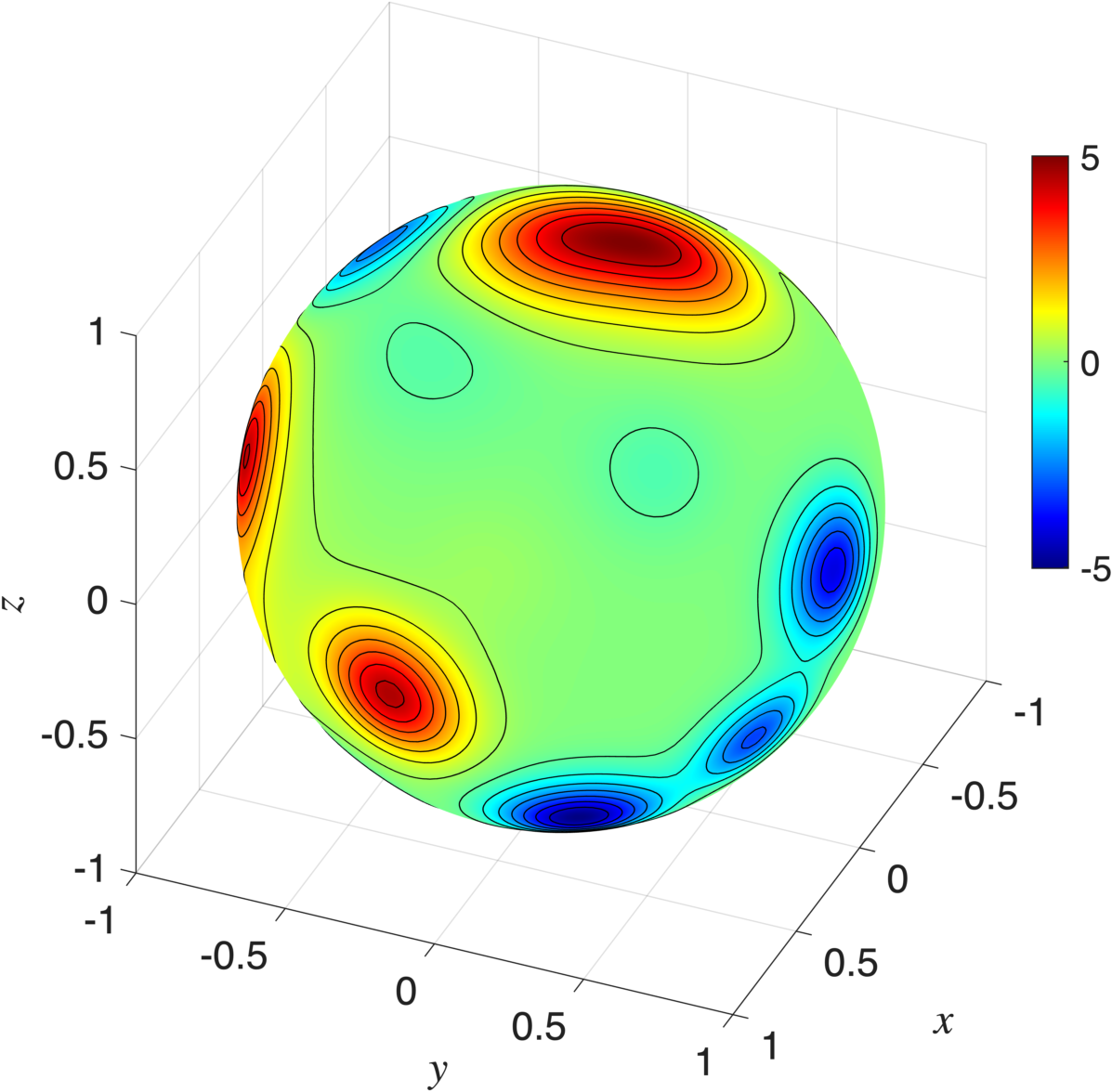} 
\caption{Target function \eqref{eq:target} used in the numerical experiments. \label{fig:target_function}}
\end{figure}
In this section, we carry out some numerical experiments illustrating the results Theorem \ref{main}.  We focus on $\M=\Sph^2$ and interpolants constructed from the restriction of the $\theta=2$ surface spline kernel, which for $x,y\in\Sph^2$ simplifies to $\Phi(x,y) = (1-x\cdot y) \log (1-x\cdot y)$. For a given target function $f:\Sph^2\rightarrow\R$ sampled at $\Xi\subset \Sph^2$, the interpolant takes the form
\begin{align}
    I_{\Xi}f = \sum_{\xi\in\Xi} a_{\xi} \Phi(\cdot,\xi) + \sum_{\ell\leq 1}\sum_{\mu=-\ell}^{\ell} c_{\ell,\mu} Y_{\ell}^{\mu}(\cdot),
    \label{eq:sbf_interp}
\end{align}
where $Y_{\ell}^{\mu}$ denotes the degree $\ell$ order $\mu$ spherical harmonic; since $\Phi$ is CPD of order 2, $\ell = 1$.  

For the interpolation points $\Xi$, we use maximum determinant points of different cardinalities $n=\#\Xi$.  These point sets are quasi-uniform, with a fill-distance that satisfies $h \sim n^{-1/2}$ and a bounded mesh ratio $\rho_{\Xi}$~\cite{Sloan_Womersley_2004}\footnote{These are available from the \texttt{spherepts} package \url{https://github.com/gradywright/spherepts}}.

We use a linear combination of randomly shifted restricted Gaussian kernels $G_{\ep}(x,y) = e^{2\ep^2(1-x\cdot y)}$ on the sphere to construct the target function:
\begin{align}
f = \sum_{y\in Y} \alpha_y G_{\ep_y}(\cdot,y),
\label{eq:target}
\end{align}
where $Y\subset\Sph^2$ are 20 scattered points, and the coefficients and shape parameters are chosen randomly as $\alpha_y\sim\mathcal{N}(0,3)$ and $\ep_y\sim\mathcal{N}(4,1)$, respectively; see Figure \ref{fig:target_function} for a visualization.  Since each shift $G_{\ep_y}(\cdot,y)$ can be expressed as $\int_{\Sph^2} \nu(\omega) \Phi(x,\omega) \diff \mu(\omega)+P$ with $P$ a spherical harmonic of degree 1, $f$ can also be expressed in a similar integral form.

\begin{figure}[htb]
\centering
\includegraphics[width=0.65\textwidth]  {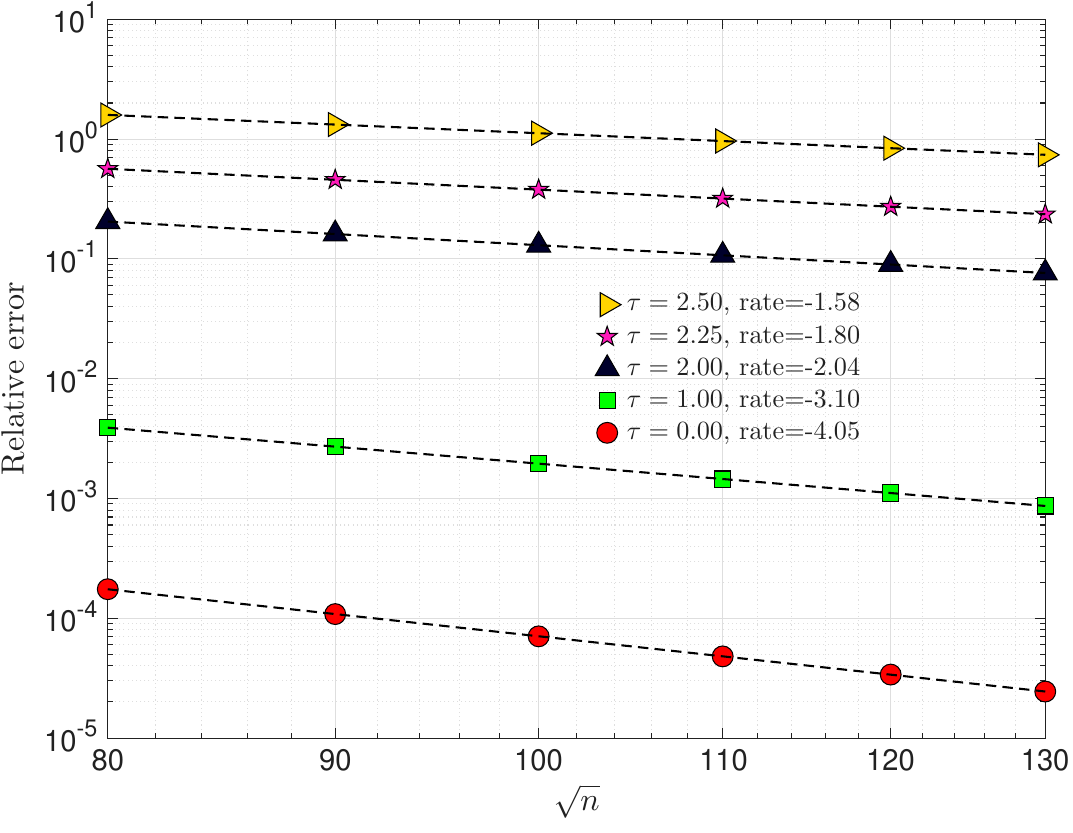}
\caption{Convergence results for the interpolant \eqref{eq:sbf_interp} of \eqref{eq:target} measured in different Sobolev norms $H^{\tau}(\Sph^2)$ \label{fig:sobolev_errors}}
\end{figure}

The Sobolev norm of order $\tau$ of a function $g:\Sph^2\rightarrow\mathbb{R}$ can be expressed as
\begin{align}
    \|g\|_{H^{\tau}(\Sph^2)}^2 = \sum_{\ell\geq 0} \sum_{\mu=-\ell}^{\ell} (1+\ell(\ell+1))^{\tau}|\widehat{g}(\ell,\mu)|^2,
    \label{eq:soblev_norm_sph}
\end{align}
where $\widehat{g}(\ell,\mu)$ are the spherical harmonic coefficients of $g$.
One approach to computing $\|I_{\Xi} f - f\|_{H^{\tau}(\Sph^2)}$ is to approximate the spherical harmonic coefficients of the error $f - I_{\Xi}$ up to some degree $L$ using a spherical harmonic transform (e.g.,\cite{schaeffer2013efficient}) and substitute them in \eqref{eq:soblev_norm_sph}. This approach requires sampling the error at a dense grid to accurately resolve the spherical harmonic coefficients up to degree $L$. 

Alternatively, we can exploit that the kernels $\Phi$ and $G_{\ep}$ are zonal and compute their spherical harmonic coefficients from their Gegenbauer expansions in the scalar variable $t = x\cdot y$~\cite{hubbert2015spherical}.  For both $\Phi$ and $G_{\ep}$, these expansions are 
known~\cite{Baxter}, 
allowing us to derive the following \textit{exact} expressions for the spherical harmonic coefficients of both \eqref{eq:sbf_interp} and \eqref{eq:target}:
\begin{align}
    \widehat{I_{\Xi}f}(\ell,\mu) &=  \begin{cases} c_{\ell,\mu} \phantom{\displaystyle \sum_{\xi\in\Xi}} & \ell \leq 1, \\ \hat{\Phi}(\ell) \displaystyle \sum_{\xi\in\Xi} a_{\xi} Y_{\ell,\mu}(\xi) & \ell > 1,
\end{cases} \label{eq:sbf_sph_coeffs}\\
\widehat{f}(\ell,\mu) &= \sum_{y\in Y}  \hat{G}_{\ep_y}(\ell) \alpha_y Y_{\ell,\mu}(y), \label{eq:target_sph_coeffs}
\end{align}
where
\begin{align*}
\hat{\Phi}(\ell) = \frac{4\pi}{(\ell+2)(\ell+1)\ell(\ell-1)}, \quad \text{and} \quad
\hat{G}_{\ep}(\ell) = \frac{2\pi^{3/2}}{\ep}\exp(-2\ep^2)I_{\ell+\frac12}(2\ep^2).
\end{align*}
Here, $I_{\ell+\frac12}$ denotes the modified Bessel function of the first kind of order $\ell+\frac12$.

Figure \ref{fig:sobolev_errors} shows the convergence results for the interpolation error in the (approximate) $H^{\tau}(\Sph^2)$ norms for different $\tau$, using \eqref{eq:sbf_sph_coeffs} and \eqref{eq:target_sph_coeffs} (see discussion below for details on this computation). The estimated rates of convergence for each $\tau$ are given in the legend. These were computed using lines of best fit to the corresponding errors. Convergence is measured against $\sqrt{n}$, since this is inversely proportional to the fill-distance $h$ of the point sets.  We see from the figure that the rates for $\tau = 0, 1, 2$ match the expected rates from 
Proposition 
\ref{P_doubling}\footnote{  
See also \cite[Theorem 3.8]{Hubbert:Morton:2004} for error estimates in $L_p$ in the doubling regime for classical interpolation on the sphere and \cite{mhaskar:etal:2009} for a different approximation scheme.},
and the rates for $\tau = 2.25$ and $2.5$ are close to those predicted by the new estimates from Theorem \ref{main}.  We note that similar convergence rates were observed as those in the figure over several different realizations of the parameters defining $f$.

\begin{figure}[htb]
\begin{tabular}{cc}
\includegraphics[width=0.45\textwidth]{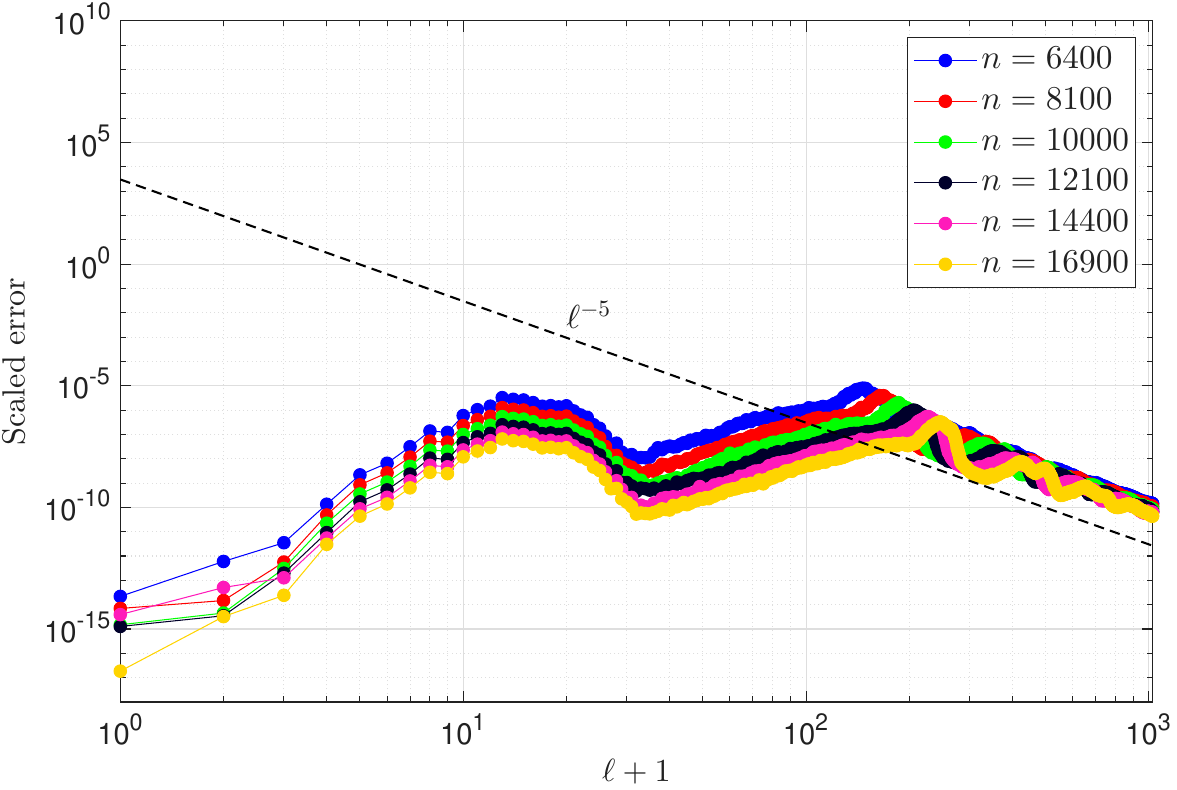} & \includegraphics[width=0.45\textwidth]{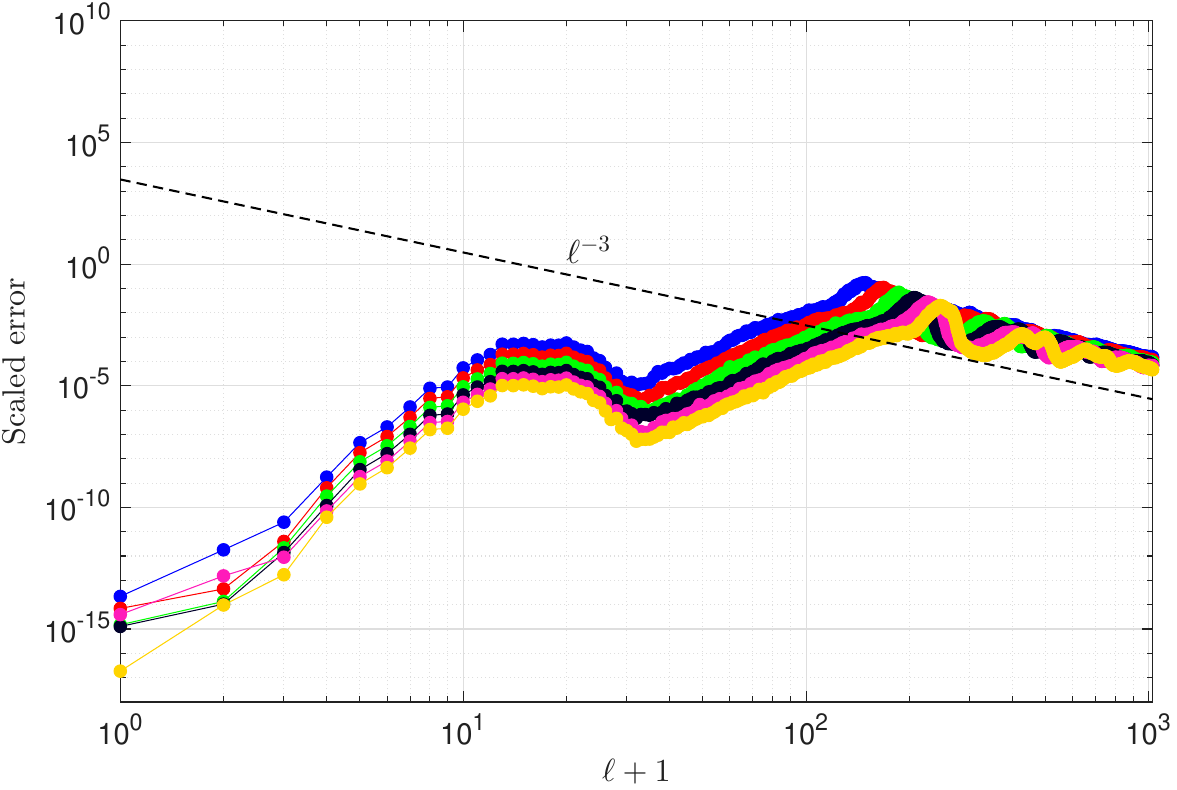} \\
(a) $\tau = 1$ & (b) $\tau = 2$ \\
\includegraphics[width=0.45\textwidth]{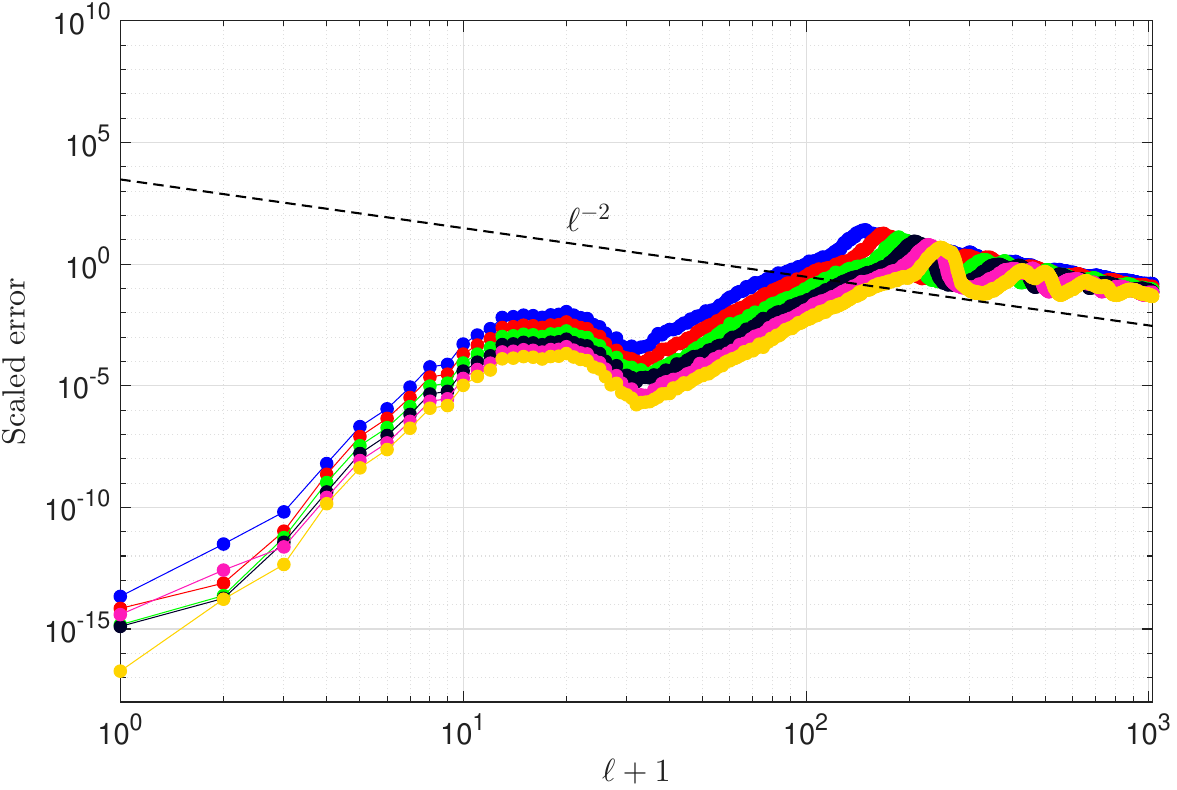} & \includegraphics[width=0.45\textwidth]{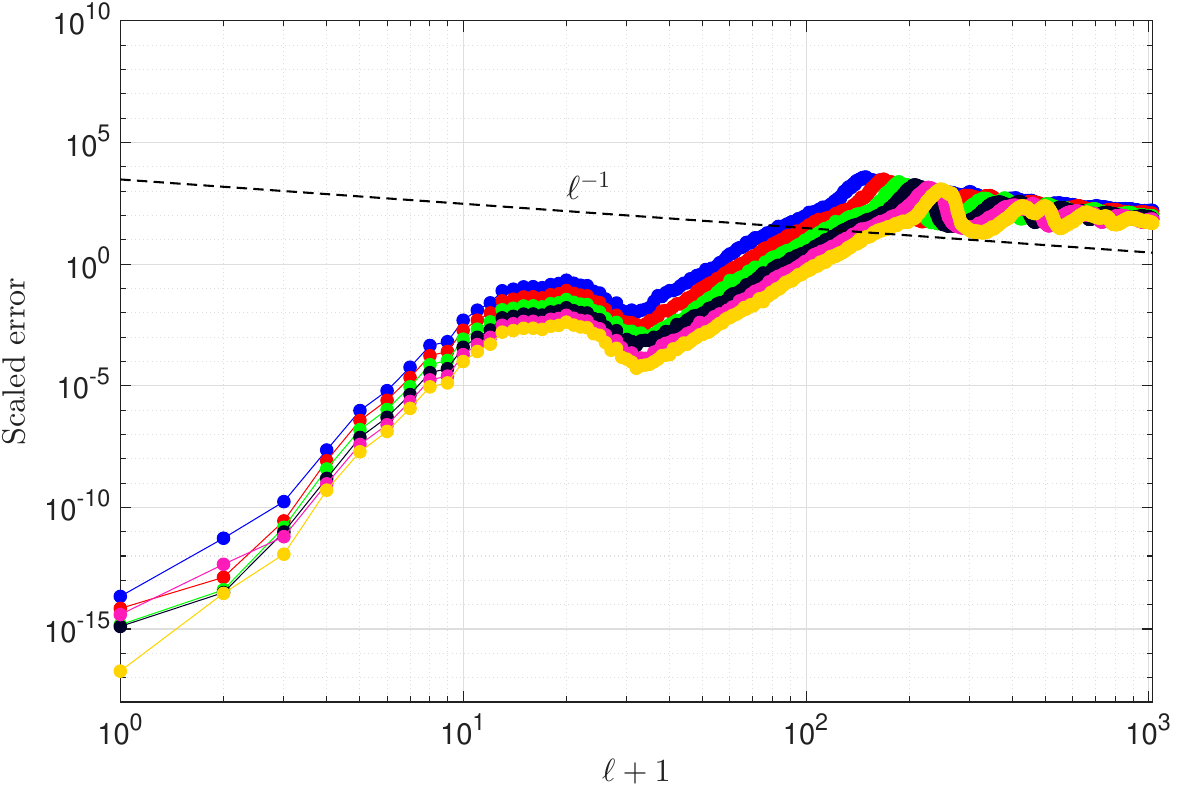} \\
(c) $\tau = 2.5$ & (d) $\tau = 3$ \\
\end{tabular}
\caption{Plots of $(1 + \ell(\ell+1))^{\tau}\sum_{\mu=-\ell}^{\ell}|\widehat{I_{\Xi}f}(\ell,\mu)-\widehat{f} (\ell,\mu)|^{2}$ for various $n=\#\Xi$.  (a)--(d) Show the results using different $\tau$. The dashed black lines corresponds to the line $\ell^{-2(4-\tau)+1}$ and give a measure of the slopes of the tails of the coefficients that are summed to approximate the $H^{\tau}(\Sph^2)$ norm \eqref{eq:soblev_norm_sph}\label{fig:norm_results2_mod}}
\end{figure}

In computing these results, we approximated the Sobolev norm \eqref{eq:soblev_norm_sph} by truncating the sum to degree $L=1024$.  This raises the question of whether sufficiently many terms were included to accurately capture the convergence rates. In Figure \ref{fig:norm_results2_mod}, we plot the scaled quantities $(1 + \ell(\ell+1))^{\tau}\sum_{\mu=-\ell}^{\ell}|\widehat{I_{\Xi}f}(\ell,\mu)-\widehat{f} (\ell,\mu)|^{2}$ for various $\tau$, along with estimates of the decay rates in the tails. The plots show that for $\tau=1,2,$ and 2.5, the tails decay at rates consistent with the convergence of the series as $L\rightarrow\infty$, while the decay rate for the tail with $\tau=3$ indicates that the series does not converge.  Although the $\tau=0$ results are omitted for brevity, we note that in this case the tail decays at the expected rate of $\ell^{-7}$.

 \section{Approximation of eigenvalues by kernel differentiation matrices}
\label{S_eigenvalues}
Considering error measured in (relatively) high order Sobolev spaces is motivated by using 
 kernel based collocation methods to solve PDEs on manifolds.
To this end, consider the differential operator $\opL:C^{\infty}(\M) \to C^{\infty}(\M)$,
which is bounded from $W_2^{t}(\M)$ to $W_2^{t-2}(\M)$.

Central to this approach is
to develop numerical differentiation by forming a 
matrix -- the differentiation matrix (DM) --
which
uses the 
kernel $(x,y)\mapsto \Phi(x,y)=\phi(x-y)$ 
and a point set $\Xi\subset \M$,
with $\#\Xi =n$,
 to represent $\opL$ as follows:
the $(\#\Xi)\times (\#\Xi)$ matrix  $\MM$ satisfies
$$(\forall u\in V_{\Xi}) \quad \MM u|_{\Xi} = (\opL u)|_{\Xi}.$$ 
The entries of $\MM$ can be succinctly described with the help of the Lagrange basis $(\chi_{\xi})_{\xi\in \Xi}$ 
of $V_{\Xi}$: 
since $u= \sum_{\xi\in\Xi} u(\xi) \chi_{\xi}$, the $(\zeta,\xi)$ entry 
of $\MM$ is
$\MM_{\zeta,\xi} = \opL\chi_{\xi}(\zeta)$.

We consider the problem of approximating spectral information (namely, eigenvalues and eigenvectors)
from the operator $\opL$ by corresponding, but more readily computed, spectral information from the differentiation matrix $\MM$.

The DM can then be used to treat time dependent parabolic problems like  $u_t +\opL u=f$ by a variety of methods (see the introduction of  \cite{HRW})
based on the semi-continuous formulation $\vec{u}_t + \MM \vec{u}=\vec{f}$  
(where $\vec{u}:[a,b]\times X\mapsto \R$). A natural assumption is that
$\opL$ has positive spectrum.
A fundamental
challenge for this problem is to ensure 
$\MM$, or its nearby perturbation $\MM^{\epsilon}$,  
have positive spectrum, too. (This is sometimes called {\em Hurwitz}stability).
A basic complication is that the matrix $\MM$ is generally not symmetric even 
when the operator $\opL$ is formally self-adjoint.

This setup has
been considered 
in \cite{HRW} in the case of spheres $\M= \Sph^{d}$, using zonal kernels, i.e.,
kernels having a 
Mercer-like
series expansion
$\Phi(x,y)=\sum_{\ell,m} c_{\ell,m} Y_{\ell}^m(x)Y_{\ell}^m(y)$   in terms of spherical harmonics,
and for operators
$\opL$   in the algebra generated by the Laplace-Beltrami operator $\Delta$.
The last condition ensures that the kernel $(x,y)\mapsto \opL_x\Phi(x,y)$ also has a Mercer-like expansion,
and that $\opL$ is symmetric on the native space of the kernel $\Phi$.
This particular sphere-based setup,
with strong and specifically sphere-based analytic properties on the kernel and operator,
guarantee positivity of the spectrum of $\MM$ along with a perturbation theory for eigenvalues; 
however, it does not explain why the spectrum of $\MM$ approximates that of $\opL$.

The issue of perturbation of eigenvalues holds in much broader generality than spheres -- for this, one can apply a version
of the Bauer-Fike theorem, along with known bounds on the condition number of the collocation matrix to get precise estimates
on the perturbation error (this is described below).
The issue of positivity of $\MM$ is less straightforward.
In this section, we give conditions  which guarantee approximation of eigenvalues and eigenvectors of $\opL$ by those of $\MM$;
these hold on spheres, but also more general manifolds, namely algebraic manifolds which are boundaries of compact regions. This provides a more general explanation of why $\MM$ has positive spectrum (without relying on the Mercer-like expansion of $\Phi$).

\subsection{Kernel differentiation matrix - positive definite case}
\label{SS_KDM_PD}
In keeping with previous sections, in the positive definite case we focus on the Mat{\'e}rn kernel:
 $(x,y) \mapsto \kernel(x,y):= \tphi(x-y)$.
Because
 $\kernel$ is positive definite,  
 the DM has factorization
$$\MM = \KK \AA,$$ 
where $\KK_{\xi,\zeta} = \opL_{\xi} \Phi(\xi,\zeta)$
is the Kansa (or asymmetric collocation) matrix
(the operator is applied to the first entry)
and $\AA=\PhiB^{-1}$ 
is  the inverse of the kernel collocation matrix  
$\PhiB_{\xi,\zeta} =\kernel(\xi,\zeta)$.

 In this case, a perturbation  $\MM \to \MM^{\epsilon}$ yields a perturbation of eigenvalues
 $|\lambda_k -\lambda_k^{\epsilon}|\le \kappa \|\MM -\MM^{\epsilon}\|$, where
 $\kappa$ is the condition number of the positive definite matrix $\PhiB^{1/2}$;
 this is a straightforward consequence of \cite[Proposition 3.1]{HRW}. 
 Furthermore, if $q$ is the separation radius, then $\kappa \le C q^{\frac{N-\theta}{2}}$ for the kernel $\tphi$ by \cite[Thm. 12.3]{Wendland_book}.
 
 In short, for computed eigenvalues $\lambda_k^{\epsilon}$, although we 
 know they are close to ideal eigenvalues $\lambda_k$ from the discretization $\MM$, we 
 do not know how these relate to the original eigenvalues $\lambda$ of $\opL$.

 \subsubsection*{Galerkin formulation} 
 To treat the spectrum of $\MM$, we recast this eigenvalue problem  as a Galerkin problem on $\Nn(\kernel)$, 
which can be treated by the method of Babu{\v s}ka and Osborn in \cite{BO}. 
 We replace $\opL u = \lambda u$ by an equivalent Galerkin problem:
 to find $u\in \Nn(\kernel),\ \lambda\in \C$ so that for all $v\in\Nn(\kernel)$
 \begin{equation}
 \label{Galerkin}
 \langle \opL u, v\rangle_{\Nn(\kernel)} = \lambda \langle u,v\rangle_{\Nn(\kernel)}
 \end{equation}
 holds.
 The right hand side uses the native space inner product, which is equivalent to 
 that of $W_2^{(\theta+d)/2}(\M)$,
  but the left involves  a bilinear form 
   $$ 
      B(u,v):=\langle \opL u, v\rangle_{\Nn(\kernel)} .
         $$
    \begin{assumption}
    \label{A_op}
$B$ is continuous, symmetric and coercive on the Hilbert space
 $H:=W_2^{(\theta+d)/2+1 }(\M)$.
 \end{assumption}
If $(\lambda,u)$ solve the variational problem (\ref{Galerkin}), then $u$ and $\opL u$ are continuous, and $\opL u = \lambda u$ holds pointwise;
 conversely, if $u$ is smooth and $\opL u = \lambda u$ holds pointwise, then (\ref{Galerkin}) holds. 
 We note that the variational problem (\ref{Galerkin}) may fail to capture eigenvalues/vectors if $\opL$ 
 does not enjoy a degree of
 elliptic regularity. Such examples are not hard to produce, see for instance the example  on $\Sph^1$
 given  in the introduction of \cite{BO2}.
 
Under this assumption, 
   the $k$th eigenvalue of $\opL$ is determined by a min-max principle as
    \begin{equation}
     \label{minimax}
    \lambda_k =\min_{\substack{W\subset {H} \ \\ \dim(W) =k} }\max_{w\in W} 
   \frac{B(w,w) 
   }{\langle w,w\rangle_{\Nn(\kernel)}}.
   \end{equation} 
 A discrete  version takes place over $V_{\Xi}$, which happens to be contained in the higher order space 
 ${H}$: 
 find $u^{(h)}\in V_{\Xi}$ and $\lambda^{(h)}\in \C$
  so that for all $ v\in V_{\Xi}$,
 $$
 \langle \opL u^{(h)}, v\rangle_{\Nn(\kernel)} = \lambda^{(h)} \langle u^{(h)},v\rangle_{\Nn(\kernel)}.
 $$
 By the minimax principle, the $k$th eigenvector $\lambda_k^{(h)}$ of the discrete problem 
 satisfies 
 \begin{equation}
 \label{discrete_minimax}
  \lambda_k^{(h)}=\min_{\substack{W\subset V_{\Xi} \ \\ \dim(W) =k} }\max_{w\in W} 
   \frac{ B(w,w) 
   }{\langle w,w\rangle_{\Nn(\kernel)}}
   \end{equation}
   and must therefore satisfy $\lambda_k^{(h)}\ge \lambda_k$. 
   For each $k\le \dim(V_{\Xi})$, there is a minimizing 
   eigenfunction $u_k^{(h)}\in V_{\Xi}$.
   \begin{lemma}
   \label{L_spectrum_equivalence}
   If $\opL$ satisfies Assumption \ref{A_op} then 
   the $k$th eigenvalue of $\MM$ is given by (\ref{discrete_minimax}).
   \end{lemma}
\begin{proof}
In this case, the stiffness matrix for the Galerkin problem on $V_{\Xi}$ using the standard basis 
is precisely $\KK$: 
since 
$u=\sum_{\xi\in X} a_{\zeta} \kernel(\cdot, \zeta)$ 
is determined from the vector 
$\vec{a} =(a_{\zeta})_{\zeta\in X}$ 
by  
$u|_{\Xi}= \PhiB \vec{a} $ 
we have
$$
B(u,\kernel(\cdot,\xi)) = 
\langle 
  \opL u, \kernel(\cdot,\xi)
\rangle_{\Nn(\kernel)}
=\opL u(\xi) = (\KK \vec{a})_{\xi}.
$$
Similarly, the mass matrix is $\PhiB$,  
and the Galerkin eigenvalue problem 
(\ref{Galerkin}) is   
$$\KK \vec{a} = \lambda \PhiB \vec{a}.$$
Since $\PhiB^{-1} \KK$ 
has the same spectrum as 
$\MM=\KK \PhiB^{-1}$,
the eigenvalues determined by 
(\ref{discrete_minimax}) 
are precisely those of the differentiation matrix 
$\MM$.
\end{proof}
  
\subsubsection*{Approximation of eigenvalues and eigenvectors
}

Letting $\EE{k}$ denote the eigenspace for the eigenvalue $\lambda_k$,
\cite{BO} shows that the quantity
$$\epsilon_{k} 
:= 
\sup_{
    \substack{u\in \EE{k}\\ \|u\|_{H}=1}
    }
\inf_{v\in V_{\Xi}} 
\|u-v\|_{H}
$$
can be used to control the 
 error between $\lambda_k$ and $\lambda_{k}^{(h)}$. 
  In particular,  \cite[Theorem 3.1]{BO} shows that
 $\lambda_{k}^{(h)}-\lambda_k\le C \epsilon_k^2$
 for some $\Xi$-independent constant $C$. We would like to point out that the result from \cite{BO} naturally forces the error of an approximation to be measured in a higher norm than the classical native space norm.
 Specifically, since 
 $H = 
 W_2^{(\theta+d)/2+1}(\M)
 $, 
 \cite[Theorem 3.1]{BO} implies
 that for each $k$, there
 exist constants $C,h_0>0$
 so that if $h<h_0$,
 \begin{equation}
\label{eq:epsk}
\lambda_{k}^{(h)}-\lambda_k\le 
C   
\inf \{
\|u-v\|_{W_2^{(\theta+d)/2+1}(\M)}^2\mid v\in V_{\Xi},
 \  u\in \EE{k}, \  \|u\|_{W_2^{(\theta+d)/2+1}}=1
 \}.
\end{equation}
Note that $C$ is independent of $h$, but may depend on $\lambda_k$.
 
 The quantity $\epsilon_{k}(h) $ also controls the error between eigenfunctions  in $\EE{k}$ and computed eigenfunctions 
 in $V_{\Xi}$ corresponding to $\lambda_k^{(h)}$ in
 that a basis $(u_j)$ can be chosen for $\EE{k}$ so that $\|u_k- u_{k}^{(h)}\|_{{H}}\le C \epsilon_k$.
In that case,
 $$
 \|u_k- u_{k}^{(h)}\|_{W_2^{(\theta+d)/2+1}(\M)}\le
 C   
\inf \{
    \|u-v\|_{W_2^{(\theta+d)/2+1}(\M)}\mid v\in V_{\Xi},
    \  u\in \EE{k}, \  \|u\|_{W_2^{(\theta+d)/2+1}}=1
 \}.
 $$
\begin{proposition}
\label{P_PD}
Suppose $\M\subset \R^N$ is a compact, closed manifold and that Assumption \ref{A_LPR} holds on $\M$. 
Suppose also that $\kernel:\M\times \M \to \R: (x,y) \mapsto  \tphi(x-y)$ is the restricted Mat{\'e}rn kernel. 
If  $\opL$ satisfies Assumption \ref{A_op}
and if $\EE{k}\subset \mathcal{A}_2$
 then the $k$th eigenvalue $\lambda_k^{(h)}$ of the differentiation matrix $\MM$ 
satisfies, for a constant $C$ 
$$0\le \lambda_{k}^{(h)}-\lambda_k\le C h^{\theta +d-2}.$$
\end{proposition}
\begin{proof}
Since $\EE{k}\subset \mathcal{A}_2$ is finite dimensional, there 
 exist constants $\gamma_A(k)$ and $\gamma_B(k)$ such that
 $$\gamma_A(k) \|u\|_{\mathcal{A}_2}\le \|u\|_{W_2^{(\theta+d)/2+1}(\M)} \le \gamma_B(k)\|u\|_{\mathcal{A}_2}, \quad \text{ for all } u\in \EE{k},$$
 where $\|u\|_{\mathcal{A}_2}= \|\nu\|_{L_2(\M)}$ for $u = \int_{\M}\Phi(\cdot, y) \nu(y) \diff \sigma(y)$.
 Hence, we can replace the normalization $\|u\|_{H}=1$ by $\|u\|_{\mathcal{A}_2}=1$ by allowing $k$ dependent constants.
 In other words, for any $u\in $ with $ \|u\|_{W_2^{(\theta+d)/2+1}(\M)}=1$, Theorem \ref{T_manifold_approximation} implies
 $$ \inf_{\nu\in V_{\Xi}} \|u-\nu\|_{W_2^{(\theta+d)/2+1}}\le C h^{(\theta+d)/2-1} \|u\|_{\mathcal{A}_2} \le C h^{(\theta+d)/2-1} /\gamma_A(k).$$
 By enlarging the constant, (\ref{eq:epsk}) gives
$$ \lambda_{k}^{(h)}-\lambda_k\le C_k  h^{\theta+d-2} $$
and the proposition follows with $C_k= C  /\gamma_A(k)$.
\end{proof}
\subsubsection*{Stronger results}
Of course if $\mathcal{A}_2=W_2^{\theta+d}(\M)$, then we may be able to say more,
specifically in case $\EE{k} \subset \mathcal{A}_2$ holds.
This can be guaranteed with a stronger conditions on $\opL$. 
\begin{assumption}
\label{A_elliptic}
The solutions to (\ref{Galerkin})  belong to $W_2^{\theta+d}(\M)$.
\end{assumption}
 This follows if
 $\opL$ is an elliptic 
pseudo-differential 
operator  with symbol in the H{\"o}rmander class $S_{1,0}^{2}$.
 Indeed,  following the discussion in 
  \cite[Chapter 7, Section 10]{Taylor2},
  if $\opL u = \lambda u$ for $u\in W_2^{s}(\M)$, then $u\in W_2^{s+2}(\M)$. 
 Thus Assumption \ref{A_elliptic} holds if $\opL$ is an elliptic differential operator with 
smooth coefficients. 
This case can be further extended
to treat operators in  local coordinates 
$$\opL u = \sum_{\alpha\le 2} a_{\alpha} D^{\alpha} u$$
with coefficients $a_{\alpha}\in C^{\theta+d-1}$
by applying \cite[Chapter 7, Theorem 2]{Evans}.

In other words,  $\EE{k}\subset W_2^{\theta+d}(\M)$.
This allows for stronger results in case of spheres and for boundaries of regions in $\R^N$.
\begin{cor}
\label{C_PD_sphere}
 Suppose
$\M=\Sph^d$
and that $\kernel:\M\times \M \to \R: (x,y) \mapsto  \tphi(x-y)$ is the restricted Mat{\'e}rn kernel.
If  $\opL$ satisfies Assumption \ref{A_elliptic}, then the $k$th eigenvalue $\lambda_k^{(h)}$ of the differentiation matrix $\MM$ 
satisfies
$$0\le \lambda_{k}^{(h)}-\lambda_k\le C_k h^{\theta +d-2}.$$
\end{cor}
\begin{cor}
\label{C_PD_boundary}
Suppose $\Upsilon$ is a bounded open subset of $\R^N$ with smooth boundary so that
$\M=\partial \Upsilon$  is an $N-1$ dimensional algebraic manifold. 
Suppose also that $\kernel:\M\times \M \to \R: (x,y) \mapsto  \tphi(x-y)$ is the restricted Mat{\'e}rn kernel with
$\theta +N\in 2\N$.
If  $\opL$ satisfies Assumption \ref{A_elliptic}, then the $k$th eigenvalue $\lambda_k^{(h)}$ of the differentiation matrix $\MM$ 
satisfies
$$0\le \lambda_{k}^{(h)}-\lambda_k\le C_k h^{\theta +N-3}.$$
\end{cor}
We point out that the growth of $k \mapsto C_k$ limits the range of eigenvalues which can be approximated with high rates. To illustrate this, we consider $\opL=-\Delta_{\M}$ the Laplace-Beltrami operator.
We consider $(\theta+d)/2+1\in \N$ and use the (equivalent) description of a Sobolev-norm $\|u\|_{W^{(\theta+d)/2+1}_{2}(\M)} \cong \|(\operatorname{id}-\Delta_{\M})^{(\theta+d)/4+1/2}u\|_{L_{2}(\M)}$. Let $\lambda_k\ge 0$ denote the eigenvalues of $-\Delta_{\M}$. In this case, we obtain for $s\ge 0$
\begin{equation*}
    \|u\|^2_{W^{s}_{2}}\cong (u, (\operatorname{id}-\Delta_{M})^s u)_{L_{2}}=\left(1+\lambda_k\right)^s\|u\|^2_{L_{2}(\M)}, \quad u \in \EE{k}.
\end{equation*}
Using for $\lambda_k\ge 1$, we obtain that $ \lambda^{s}_k \le (1+\lambda_k)^s \le 2^s \lambda^{s}_k$.
Hence, if $\mathcal{A}_2= W_2^{\theta+d}(\M)$, then we obtain from \eqref{eq:epsk}
\begin{align*}
    \lambda_{k}^{(h)}-\lambda_k&\le C_k \inf \{
    \|u-v\|^2_{W_2^{(\theta+d)/2+1}(\M)}\mid v\in V_{\Xi},
    \  u\in \EE{k}, \  \|u\|_{W_2^{(\theta+d)/2+1}}=1
 \}\\
 &\le C_k \inf \{
     h^{\theta+d-2} \| u\|^2_{W^{\theta+d}_{2}(\M)}\mid v\in V_{\Xi},
    \  u\in \EE{k}, \  \|u\|_{W_2^{(\theta+d)/2+1}}=1
 \}\\
 &\le  C_k \inf \{
     h^{\theta+d-2} 2^{\theta+d} \lambda^{(\theta+d)/2-1}_{k}\|u\|^2_{W_2^{(\theta+d)/2+1}} \mid v\in V_{\Xi},
    \  u\in \EE{k}, \  \|u\|_{W_2^{(\theta+d)/2+1}}=1
 \}\\&\le C_k h^{\theta+d-2} 2^{\theta+d} \lambda^{(\theta+d)/2-1}_{k}=C_k  2^{\theta+d} (h^2 \lambda_k)^{(\theta+d)/2-1}.
\end{align*}
This suggest a coupling of $h$ to $k$ at least for those values of $k\le K$, where  $C=\max_{1\le k \le K}C_k$ is bounded independently of $N$ or $h$.
  
 \subsection{Conditionally positive definite case} 
 In case  $\Phi$ is conditionally positive definite, with auxiliary space $\varPi$,
having dimension $J:= \dim( \varPi)$ 
and basis 
$\{\varphi_j\mid j\le J\}$ for $\varPi$,
 the differentiation matrix has the form
$\MM = \KK \AA  
  + 
  \mathbf{ Q}
  \BB$
where 
   $\mathbf{Q} = \bigl(\opL \varphi_{j}(\xi)\bigr)_{j\le J,\xi\in {\Xi}}$
   is a Vandermonde matrix and
   the matrices 
 $\AA$ and $\BB$
 consist of coefficients of the Lagrange functions: $\chi_{\xi} = \sum_{\zeta\in {\Xi}} \AA_{\zeta,\xi} \Phi(\cdot,\zeta)+
 \sum_{j\le J} \BB_{j,\xi} \varphi_j$.
  These occur as solutions 
 to the augmented kernel collocation problem described in \cite{HRW}.
This  extra complication may be simplified somewhat  if the auxiliary space  
 is invariant under $\opL$.
 If so, then $\mathbf{Q} =\mathbf{P} \mathbf{\Lambda}$ for some block-diagonal
 matrix $\mathbf{\Lambda}$ and Vandermonde matrix
 $\mathbf{P} = \bigl(\varphi_{j}(\xi)\bigr)_{j\le J,\xi\in {\Xi}}$.
 Thus,
 \begin{align}
 \MM = \KK \AA  
  + 
\mathbf{P}\LLambda \BB
\label{eq:CPDDM}
\end{align}
 This scenario is common in practice, even when $\Phi$ is strictly positive definite,
  because $\MM$ is exact on elements of 
  $\varPi$.

\begin{example} If $\M=\mathbb{S}^d$ and  $\opL$ is a suitable function of the Laplace-Beltrami operator
and the kernel $\Phi$ 
is a standard spherical basis function (SBF),
then 
$\varPi=\mathcal{P}_{M-1} (\mathbb{S}^d) $ consists of spherical harmonics, and is thus invariant under $\opL$.
\end{example}

\begin{example} If $\M$ is a generic compact manifold and $\mathrm{Null}(\opL) = \mathcal{P}_0(\M)$, we may
wish to consider $\Phi$ to be CPD of order $0$ (even if it is strictly positive definite). 
Note that in this case, 
$\MM = \KK \AA  $, since $ \LLambda =0$
and  so $\mathbf{ Q}   \BB = \mathbf{P} \LLambda   \BB = 0$.
This has the advantage of consistency:   $\MM$ has $0$ as its smallest eigenvalue (matching that of $\opL$).
\end{example}
In the case of conditional positive definiteness, we thus make an assumption which has no analog in the positive definite case.
\begin{assumption}\label{ass:diag}
We assume that   
$$\opL|_{\varPi}:\varPi\to \varPi$$
is diagonalizable, and 
$\{\varphi_j\mid j\le J\}$ is an eigenbasis for $\opL|_{\varPi}$.
\end{assumption}
As a consequence, we obtain that matrix $\LLambda $ is diagonal in the differentiation matrix $\MM = \KK \AA  + \mathbf{P}\LLambda \BB$ .

\subsubsection*{Galerkin formulation} 
 Consider the Hilbert spaces 
 $$ \Nn_{\ord}(\Phi)/\mathcal{P}_{\ord-1} (\M)
 \qquad
 \text{ and }
 \qquad 
 H := W_2^{(\theta+d)/2+1}(\M)/\mathcal{P}_{\ord-1} (\M).$$
 Define
 $$B(u,v) = \langle \opL u,v\rangle_{\Nn_{\ord}(\Phi)}.$$ 
 Next, we need the analogue of Assumption \ref{A_op} for the conditionally positive definite case.
 \begin{assumption}
 \label{A_op_cond}
 $B:H\times H \to \R $ is continuous, symmetric and coercive on $H$.
 \end{assumption}
 Recall that $\mathcal{P}_{\ord-1}(\M)$ is the nullspace of the semi-inner product on $\Nn_{\ord}(\Phi)$, 
 so it is also the nullspace of $B$.
  Find $u\in \Nn_{\ord}(\Phi),\ \lambda\in \C$ so that for all $v\in\Nn_{\ord}(\Phi)$
 \begin{equation}\label{Galerkin_CPD}
 \langle \opL u, v\rangle_{\Nn_{\ord}(\Phi)} = \lambda \langle u,v\rangle_{\Nn_{\ord}(\Phi)}
 \end{equation}
 holds. The discrete version over $V_{\Xi}$ reads now: 
 find $u^{(h)}\in V_{\Xi}$ and $\lambda^{(h)}\in \C$
  so that for all $ v\in V_{\Xi}$
 \begin{equation}\label{discrete_Galerkin}
 \langle \opL u^{(h)}, v\rangle_{\Nn_{\ord}(\Phi)} = \lambda^{(h)} \langle u^{(h)},v\rangle_{\Nn_{\ord}(\Phi)}
 \end{equation}
 holds.

By Assumption \ref{A_op_cond}, we have as in the positive definite case that
   the $k$th eigenvalues for both the continuous and discrete problem
   are determined by a min-max principle as
    \begin{equation}
    \label{discrete_minimax_CPD}
    \lambda_k =\min_{\substack{W\subset {H} \ \\ \dim(W) =k} }\max_{w\in W} 
   \frac{B(w,w) 
   }{\langle w,w\rangle_{\Nn_{\ord}(\Phi)}} \quad \text{and} \quad 
   \lambda_k^{(h)}=\min_{\substack{W\subset V_{\Xi} \ \\ \dim(W) =k} }\max_{w\in W} 
   \frac{ B(w,w) 
      }{\langle w,w\rangle_{\Nn_{\ord}(\Phi)}}.
\end{equation}
Therefore we automatically have $\lambda_k^{(h)}\ge \lambda_k$. 
We use the Lagrange basis $(\chi_{\xi})$ for $V_{\Xi}$ to describe the stiffness and mass matrices.
The stiffness matrix 
has $\xi,\eta$ entry
$B(\chi_{\xi} ,\chi_{\zeta} ) 
= \sum \sum \AA_{\nu,\zeta} 
\opL \Phi(\nu,\eta) \AA_{\eta,\xi}
$; 
i.e.,
it is $\AA^T \KK\AA = \AA\KK\AA$, since $\AA$ is symmetric. 
Similarly, the mass matrix has $\xi,\eta$ entry
$\langle \chi_{\xi} ,\chi_{\zeta} 
\rangle_{\Nn_{\ord}(\Phi)} = \sum \sum \AA_{\nu,\zeta}  \Phi(\nu,\eta) \AA_{\eta,\xi}$;
i.e., the mass matrix is 
$\AA^T \PhiB \AA 
=
\AA \PhiB \AA = \AA$,
since  $\mathbf{I}_{\Xi} = \bigl(\chi_{\xi}(\zeta)\bigr)_{\xi,\zeta} = \PhiB \AA + \mathbf{P}\BB$, 
and 
$\AA\mathbf{P} =\mathbf{0}_{J,M}$.

In short, the solutions $(\lambda^{(h)}, u^{(h)})\in \R\times V_{\Xi}$ 
to (\ref{discrete_Galerkin}) are in one to one correspondence  to the eigenvalues which satisfy
\begin{equation}
\label{eq:discrete_GalerkinMat}
\AA \KK\AA \vec{u} = \lambda^{(h)} \AA \vec{u}
\end{equation}
via $u^{(h)} = \sum \vec{u}_{\xi} \chi_{\xi}$.

\subsection*{Comparison to eigenvalues of the differentiation matrix}
Finally, we aim at a variant of 
Lemma 
\ref{L_spectrum_equivalence} in 
the conditionally positive 
definite case which is much 
harder due the the more 
complicated structure of the 
differentiation matrix $\MM$.
The numerical linear algebra is 
taken from \cite{HRW} and we 
recall it for the reader's 
convenience.

\begin{lemma}
 Given Assumptions \ref{ass:diag} and \ref{A_op_cond}, the differentiation matrix $$\MM
 \sim\begin{pmatrix}
\mathbf{\Lambda}& \mathbf{R}\\\mathbf{0}&\mathbf{\Theta}\end{pmatrix}$$
is similar to an upper block triangular matrix where the diagonal blocks are themselves  diagonal matrices 
$\mathbf{\Theta}$ and $\mathbf{\Lambda}$.

If $\vartheta \in \sigma(\mathbf{\Theta}) \setminus \sigma(\mathbf{\Lambda})$, then $\vartheta$ is an eigenvalue of $\MM$
and appears as a solution of \eqref{eq:discrete_GalerkinMat}. On the other hand,
if $\vartheta\in \sigma(\mathbf{\Lambda})\cap \sigma(\mathbf{\Theta})$, then it is a possibly 
generalized eigenvalue of $\MM$, with $\MM u = \vartheta u +v$ and $v\in N_{\AA}$.
Also in that case, $(u,\vartheta)$ is an eigenpair of \eqref{eq:discrete_GalerkinMat}. 
\end{lemma}
\begin{proof}
The block triangularization  is presented in  \cite[Lemma 4.1]{HRW}.
If $\vartheta\in\sigma(\mathbf{\Theta})$ and $\vec{u}$  satisfies $\MM \vec{u} = \vartheta \vec{u} +\vec{v}$
with $\vec{v} \in N_{\AA} $,
then 
$\MM=\KK\AA +\mathbf{P} \mathbf{\Lambda} \BB$ and
 $N_\AA=R_\mathbf{P}$ imply
$$\AA \KK\AA \vec{u}  = \AA(\MM - \mathbf{P} \mathbf{\Lambda} \BB)\vec{u} = \AA \MM\vec{u} = \vartheta \AA \vec{u},$$
so $(\vartheta,\vec{u})$ is a solution to  (\ref{eq:discrete_GalerkinMat}).
\end{proof}
It follows that the $n-J = \dim (V_{\Xi}/\varPi)$ solutions to \eqref{eq:discrete_GalerkinMat}
correspond to the generalized eigenpairs of $\MM$ which are not in $\varPi(\Xi)$.

We are now in a position to give error estimates for the $\lambda_k^{(h)} - \lambda_k$ as in section \ref{SS_KDM_PD}.
\begin{proposition}
\label{P_CPD}
Suppose  the following assumptions hold:
\begin{itemize}
\item $\M\subset \R^N$ is a compact, closed manifold and that Assumption \ref{A_LPR} holds on $\M$, 
\item  $\kernel:\M\times \M \to \R: (x,y) \mapsto  \rbf(x-y)$ is 
conditionally positive definite with respect to  $\mathcal{P}_{m-1}(\M)$,
with $m\ge 0$ in case of  the restricted Mat{\'e}rn kernel
and $m\ge \lfloor \theta/2\rfloor+1$ for the surface spline kernel,
\item  $\opL$ satisfies  Assumption \ref{A_op_cond} 
with $H =W_2^{(\theta+d)/2+1}(\M)/\mathcal{P}_{\ord-1} (\M)$
as well as
 Assumption \ref{ass:diag} 
 with $\varPi = \mathcal{P}_{m-1}(\M)$.
 \end{itemize}
If $\lambda_k$ is the $k$th solution 
to (\ref{Galerkin_CPD}) and 
 if $\EE{k}\subset \mathcal{A}_2$,
 then there are constants $C$ and $h^*>0$ so that if 
 $\Xi $ is sufficiently dense in $\M$, namely with $h(\Xi,\M) \le h^*$,
  then there is
 an eigenvalue $\lambda_k^{(h)}$  of $\MM$,
 specifically a diagonal element of $\mathbf{\Theta}$, %
so that 
$$0\le \lambda_{k}^{(h)}-\lambda_k\le C  h^{\theta +d-2}.$$
\end{proposition}
\begin{figure}[htb]
    \centering
    \begin{tabular}{c}
    \includegraphics[width=0.6\linewidth]{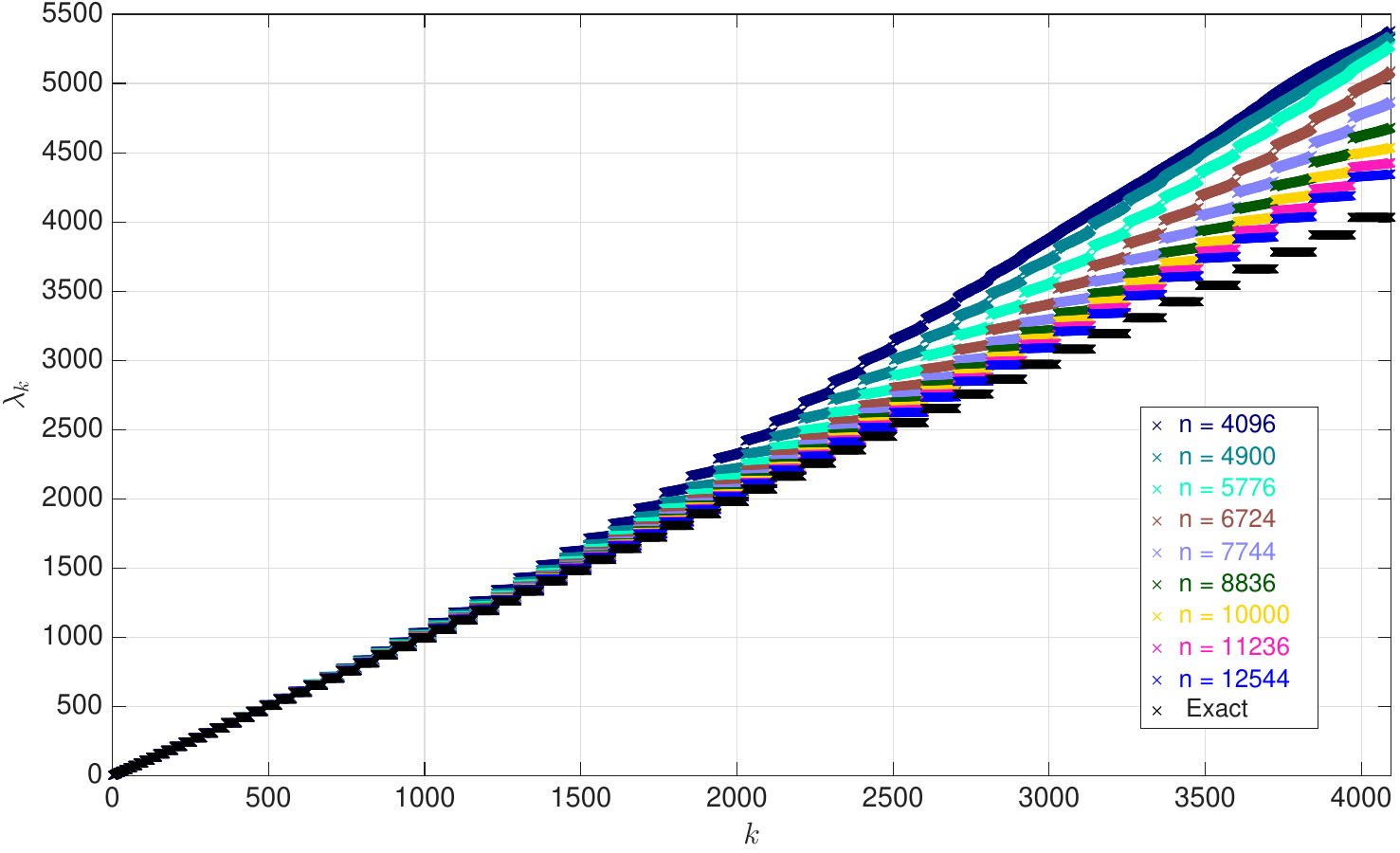} \\ (a) $\theta=4$ restricted surface spline \\
    \includegraphics[width=0.6\linewidth]{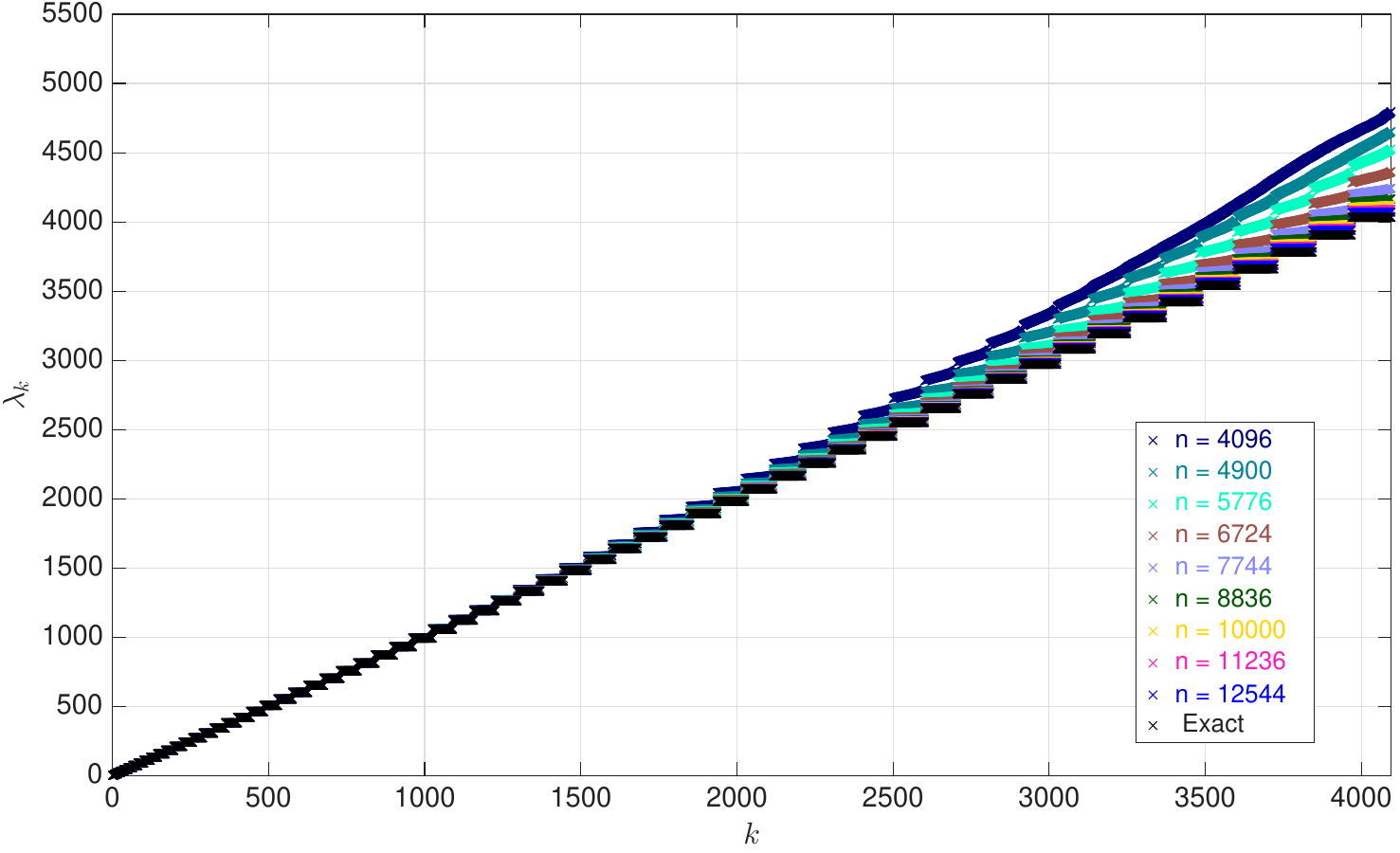} \\ (b) $\theta = 6$ restricted surface spline
    \end{tabular}
    \caption{Comparison of of the first 4096 eigenvalues, $\lambda_k^{(h)}$, of the differentiation matrix for different $N$, computed using the restricted surface splines of orders (a) $\theta=4$ and (b) $\theta = 6$. The solid black $\times$'s mark the exact eigenvalues $\lambda_k$ (with correct multiplicity) of the Laplace-Beltrami operator on the sphere}
    \label{fig:spectrumvsN}
\end{figure} 
As in the positive definite case, we have the following two corollaries for which
$\EE{k}\subset \mathcal{A}_2$ holds automatically and 
so this hypothesis 
can be  
omitted. 
\begin{cor}
\label{Sphere_CPD}
 Suppose
the hypotheses of Proposition \ref{P_CPD} hold, and that $\M=\Sph^d$.
Then
if $\lambda_k$ is the $k$th solution 
to (\ref{Galerkin_CPD})  
  there are constants $C$ and and $h^*>0$ so that if 
 $h=h(\Xi,\Sph^d)\le h^* $
  there is
 an eigenvalue $\lambda_k^{(h)}$  of $\MM$,
so that 
$$0\le \lambda_{k}^{(h)}-\lambda_k\le  C  h^{\theta +d-2}.$$
\end{cor}

\begin{figure}[thb!]
    \centering
    \begin{tabular}{c}
    \includegraphics[width=0.6\linewidth]{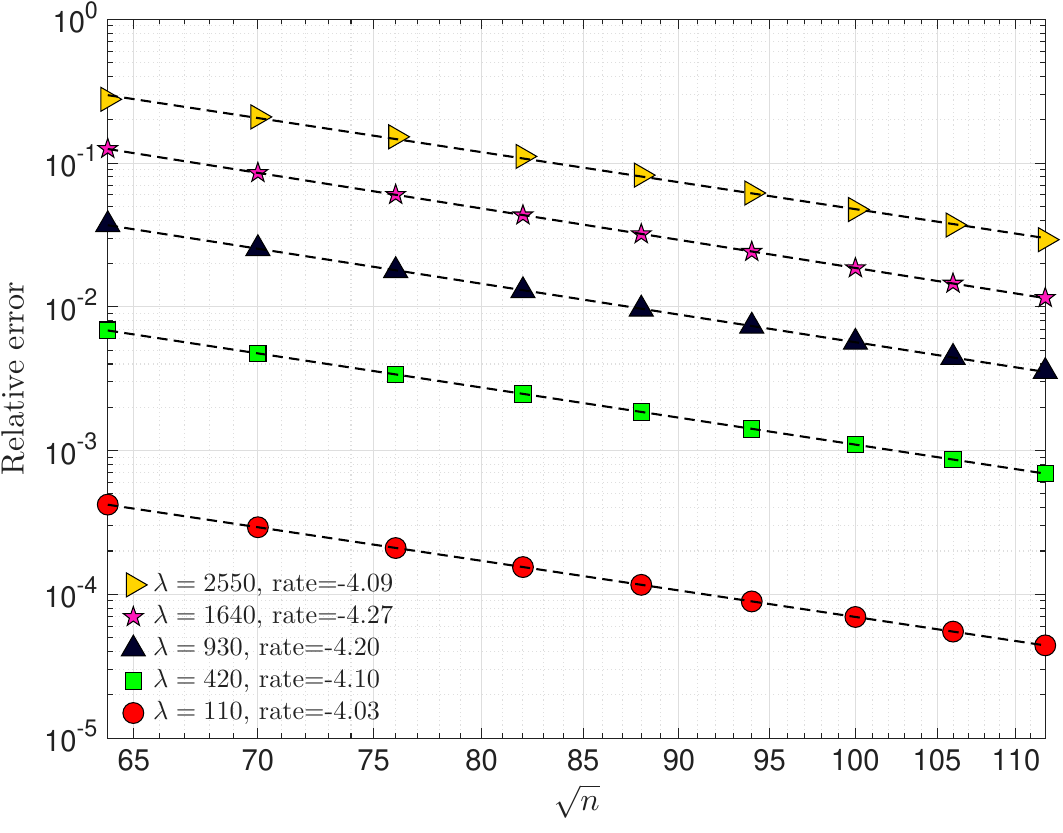} \\ (a) $\theta=4$ restricted surface spline \\
    \includegraphics[width=0.6\linewidth]{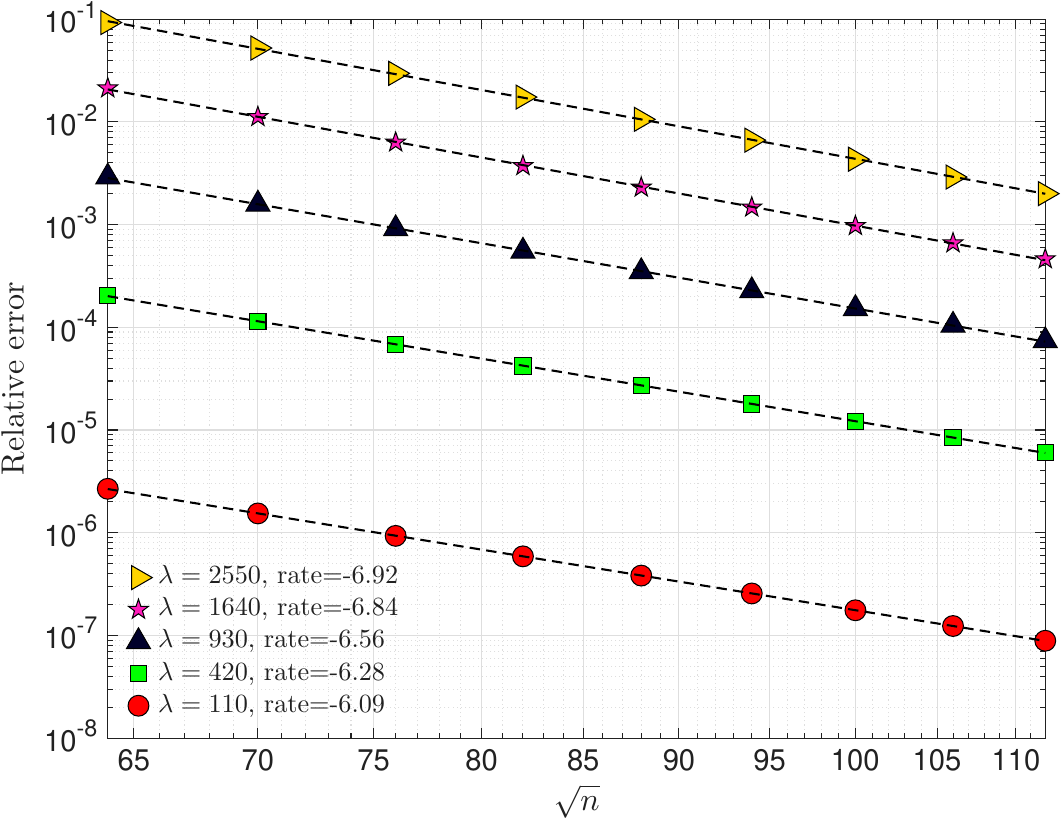} \\ (b) $\theta = 6$ restricted surface spline
    \end{tabular}
    \caption{Convergence results for approximating various eigenvalues of the Laplace-Beltrami operator on $\Sph^2$ using DMs computed from restricted surface splines of orders 2 (a) and 3 (b). The dashed black lines are the lines of best fit to the data, and the estimated rates of convergence from these lines are given in the legend.}
    \label{fig:EgvlConvgvsN}
\end{figure} 

\begin{cor}
\label{Boundary_CPD}
 Suppose
 the hypotheses of Proposition \ref{P_CPD} hold, and moreover
that $\M=\partial \Upsilon$ 
for a  bounded open subset  $\Upsilon$ of $\R^N$
 and $\theta +N\in 2\N$.
Then
if $\lambda_k$ is the $k$th solution 
to (\ref{Galerkin_CPD})  
   there are constants $C$ and and $h^*>0$ so that if 
 $h=h(\Xi,\Sph^d)\le h^* $
  there is
 an eigenvalue $\lambda_k^{(h)}$  of $\MM$,
so that 
$$0\le \lambda_{k}^{(h)}-\lambda_k\le  C  h^{\theta +N-3}.$$
\end{cor}

\subsection{Numerical Experiments\label{sec:egvl_numerics}}

In this section, we present numerical experiments illustrating the convergence results from the previous section. While the spectrum of the Laplace-Beltrami operator on $\M=\Sph^2$ is known explicitly, it provides an ideal setting for validating the spectral convergence predicted by our analysis. To this end, we compare the exact eigenvalues of the Laplace-Beltrami operator with those of the DMs constructed from the restricted surface spline kernel \eqref{eq:CPDDM}. As in section~\ref{sec:numerics_higher_norms}, we use maximum determinant point sets in all experiments.

Figure \ref{fig:spectrumvsN} compares the exact eigenvalues for $1\leq k \leq 4096$ with those of the DMs constructed for point sets of various sizes $n$. Results for both the  $\theta = 4$ and $\theta = 6$ surface spline kernels (which are CPD of orders 3 and 4, respectively) are included to illustrate the improved approximation with higher-order kernels.  In both cases, we see that the lower portion of the spectrum is approximated more accurately, consistent with the convergence theory.

The convergence of the eigenvalue approximations is examined in more detail in Figure \ref{fig:EgvlConvgvsN}, which displays the relative errors vs $\sqrt{n}$ for a select number of eigenvalues.  For each eigenvalue, we compute the maximum relative error in approximation, taking into account the eigenvalues algebraic multiplicity.  For the $\theta=4$ restricted surface spline results (Figure \ref{fig:EgvlConvgvsN}(a)), the estimated convergence rates closely match the predicted  fourth-order rate.  For $\theta=6$ (Figure~\ref{fig:EgvlConvgvsN}(b)), the smaller eigenvalues exhibit the expected sixth-order convergence, whereas the larger eigenvalues appear to converge at a higher rate. This behavior is consistent with the theory in~\cite{BO}, which predicts that the pre-asymptotic regime for approximating $\lambda_k$ grows with $k$.  Consequently, for the larger eigenvalues considered here, the computations may not yet have reached the asymptotic regime in which the predicted sixth-order convergence is observed.
 
\appendix
\section{Discrete approximation of \texorpdfstring{$d$}{}-measures}
\label{S_measures}
In this section we consider, 
for a set $\Omega$ and  measure $\sigma$ satisfying Assumption \ref{ball_comparison}
and a function  $\nu\in L_p(\sigma)$ defined on $\Omega$,
a method for approximating the push-forward $\tilde{\nu}$ of $\nu\,\diff \sigma$ 
(considered as a measure on $\R^N$) by using a local polynomial reproduction of $\Omega$.
The approximant is the discrete measure
$$\tilde{\nu}_{\Xi} = \sum_{\xi\in \Xi} \Bigl(\int_{\Omega} a(\xi,y) \nu(y)\diff \sigma(y)\Bigr) \delta_{\xi}$$
supported on $\Xi$.
The error $\tilde{\nu} -\tilde{\nu}_{\Xi}$ measured in Besov norms
can be directly estimated using  the norms and  semi-norms given by the spectral partition of unity
in the definition in section \ref{SS_smoothness}.

There are two main tools used in this section. The first is a consequence of Assumption \ref{ball_comparison} is that we can estimate certain integrals over $\Omega$ by
way of a series test.
In particular, if $f:\Omega\to [0,\infty)$ satisfies $f(x) = g(a|x-y|)$ for some $y\in \R^N$ and $a>0$ and 
decreasing function $g:[0,\infty)\to[0,\infty)$ for which $\sum 2^{jd} g(2^j)<\infty$, then by decomposing 
$\Omega$ en annuli, with $A_j = \{x\in \Omega\mid 2^{j-1}\le a |x-y| \le2^ j \}$, we obtain
\begin{equation}
\label{integral_test}\int_{\Omega} f(x)\diff \sigma(x) \le \omega_{\Omega} \sum_{j=1}^{\infty}  \left(\frac{2^j}{a}\right)^d g(2^{j-1})\le C\ a^{-d}  \bigl(g(0)+\sum_{j=1}^{\infty}  2^{jd}g(2^{j-1})\bigr),
\end{equation}
with $C = 2^d \omega_{\Omega}$.

The second tool is the following observation:
for a general Schwartz function $f$, we
have $(\tilde{\nu} - \tilde{\nu}_{\xi})*f = \int_{\Omega} \nu(y) \left[ f(x-y) - \sum_{\xi\in\Xi} a(\xi,y) f(x-\xi) \right] \diff \sigma(y).$
Convolving the error with a function $g$  leads to the corresponding approximation problem for $g$ using the LPR.
This is treated by \cite[Lemma 4]{HR-Extending}, which states that for any $m\le L$, that
\begin{equation}
\label{LPR_error}
\Bigl| f(x-y) - \sum_{\xi\in\Xi} a(\xi,y) f(x-\xi) \Bigr| \le  C h^{m}  \max_{|\beta|=m} \|D^{\beta} f(x-\cdot)\|_{B(y,Kh)}
\end{equation}
holds. The freedom to choose $m$ less than $L$ is important (the LPR also reproduces $\mathcal{P}_{m-1}$ for $m\le L$). 

We can use this to treat the $L_p$  (low frequency) component of the Besov norm by the following theorem.
\begin{lemma}
There exists a constant $C$ (depending on the family $\HF, \LF$, the constant $K$ in Assumption \ref{A_LPR}), 
 and the constant $\omega_{\Omega}$ in Assumption \ref{ball_comparison}) 
so that
for $1\le p\le \infty$ there 
$$\| \LF*(\tilde{\nu}- \tilde{\nu}_{\Xi})\|_{L_p(\R^N)}< C h^{L} \|\nu\|_{L_p(\sigma)}
$$
\end{lemma}
\begin{proof}
Applying (\ref{LPR_error})  to $\LF$ with $m=L$ gives
$$\Bigl| 
\LF(x-y) - 
\sum_{\xi\in\Xi} a(\xi,y) \LF(x-\xi) 
\Bigr| 
\le  
C h^{L} (1+ |x-y|)^{-(N+1)}$$
where $C$ incorporates 
the  constant 
in  (\ref{LPR_error}) and the 
Schwartz seminorm of $\LF$.
The result follows by Schur's test, 
first by
integrating this expression with 
respect to $x$ (over $\R^N$) 
and $y$ 
(over $\Omega$,
using (\ref{integral_test}) with 
$g(x) = (1+|x|)^{-(N+1)}$ and $a= 1$)
to get endpoint bounds
and then using interpolation to get the result for general $p$.
\end{proof}
To treat the semi-norms
$\| \HF_k*(\tilde{\nu}- \tilde{\nu}_{\Xi})\|_{L_p(\R^N)}$,
 we need 
a more precise, $k$-dependent 
estimate  on
$\left| \HF_k(x-y) - \sum_{\xi\in\Xi} a(\xi,y) \HF_k(x-\xi) \right|$. The challenge is that this involves two 
spatial resolutions:the support radius $2Kh$ of the LPR and the 
scaling $2^{-k}$ of the test function $\psi_k$.
\begin{lemma}
There is $C$ so that for every $k\in\N$, the inequality
 \begin{equation}
 \label{ptwise_convo_bound}
 \Bigl| \HF_k(x-y) - \sum_{\xi\in\Xi} a(\xi,y) \HF_k(x-\xi) \Bigr| 
 \le
 C2^{kN} (h2^{k})^{\omega} 
 \begin{cases}
 1 & \text{for all } x,y, \\
\left(2^k{|x-y|}\right)^{-(N+1)} 
& \text{if } |x-y|>2Kh
\end{cases}
 \end{equation}
 holds for any real $ \omega \in[0, L]$.
\end{lemma}

\begin{proof}We begin with the observation: $\left| \HF_k(x-y) - \sum_{\xi\in\Xi} a(\xi,y) \HF_k(x-\xi) \right| \le
 C2^{kN}(h2^{k})^{m} $,
 which follows from a direct application of (\ref{LPR_error}) using the chain rule, the fact that
each  function $D^{\beta}\HF$ is bounded, and $\sum_{\xi\in\Xi}| a(\xi,y)|\le 2$. 
This extends naturally to non-integer cases $\omega$ which lie in the interval $[0,L]$.
In particular, this treats the first case, $|x-y|<2Kh$.

\smallskip
 
 A finer estimate is possible if we consider scaling 
 and applying (\ref{LPR_error}). 
 In that case, consider  the map
 $$\mathring{a}: 2^k\Xi\times 2^k\Omega \to \R : (2^k\xi, 2^k y)\mapsto
  a(\xi, y).$$
  By dilation invariance of $\mathcal{P}_{L}(\R^N)$ this is a local polynomial
  reproduction of degree $L$ on $2^k \Omega$
  with support 
   $\mathrm{supp}\Bigl( \mathring{a}(\cdot,  Y)\Bigr)
 \subset B(Y, 2^k K h)$.
  Indeed, for $p\in \mathcal{P}_{L}(\R^N)$ and $Z\in 2^k \Omega$, the equation
  $$\sum_{2^k\xi\in 2^k\Xi} \mathring{a}(2^k \xi, Z) p(2^k \xi) = \sum_{\xi\in \Xi}  a(\xi, 2^{-k} Z) p(2^k \xi)
   = p\bigl(2^k (2^{-k}Z)\bigr) = p(Z)$$
   holds.

 By (\ref{LPR_error}),
 applying $\mathring{a}$
to $\HF$, reveals
 for $X\in \R^N$ and $Y\in 2^k\Omega$ that
$$ \Bigl| \HF(X-Y) - \sum_{\xi\in\Xi} \mathring{a}(2^k \xi,Y) \HF(X-2^k \xi) \Bigr| \le  C  (h2^{k})^{m} 
 \max_{|\beta|=m} \|D^{\beta} \HF (X-\cdot)\|_{B(Y, 2^k K h)}.
 $$ 
 If $|X-Y|> 2 (2^k K h)$, and if $Z\in  B(Y, 2^k K h)$, then
 $|X -  Z| > \frac{1}{2} |X-Y|$. 
 It follows from $\HF\in \mathcal{S}(\R^N)$  that
 for any $\nu$ there is a constant $C$ so that
 \begin{eqnarray*}
 \Bigl| \HF(X-Y) - \sum_{\xi\in\Xi} \mathring{a}(2^k \xi,Y) \HF(X-2^k \xi) \Bigr| 
& \le&
 C  (h2^{k})^{m} 
 \min_{Z\in B(Y, 2^k Kh)}
 |X-Z|^{-(N+1)}  \\
&\le&
2^{N+1} C  (h2^{k})^{m} 
 |X-Y|^{-(N+1)}.  
 \end{eqnarray*}
The integer value $m$ can be replaced by any real $\omega \in[0, L]$.

Finally, by changing variables $X,Y$ to $x = 2^{-k} X$ and $y = 2^{-k} Y$,
the definition  $\HF_k = 2^{kN} \HF(2^k\cdot)$ gives
$\left| \HF_k(x-y) - \sum_{\xi\in\Xi} a(\xi,y) \HF_k(x-\xi) \right| \le  C 2^{kN} (h2^{k})^{\omega} 
 (2^k {|x-y|})^{-(N+1)}.  $ Combining this with the earlier bound and increasing the constant gives the result.
 \end{proof}
\begin{lemma}
There exists a constant $C$ (depending on the family $\HF, \LF$, the constant $K$ in Assumption \ref{A_LPR}, 
 and the constant $\omega_{\Omega}$ in Assumption \ref{ball_comparison}) 
so that
for $1\le p\le \infty$ there 
$$
\| \HF_k*(\tilde{\nu}-\tilde{\nu}_{\Xi})\|_{L_p(\R^N)}< C \|\nu\|_{L_p(\sigma)}\begin{cases}
2^{k\frac{N-d}{p'}} (h2^k)^{\omega}
&2K h2^k <1\\
2^{k \frac{N}{p'} }   h^{\frac{d}{p'}} 
&2K h2^k \ge 1.\\
\end{cases}$$
\end{lemma}
\begin{proof}
We consider two cases, depending on the relative size of $h$ and $k$, although both
use inequality (\ref{ptwise_convo_bound}).

{\em  Case 1:} If $2K h<2^{-k}$, then 
we can rewrite the bound (\ref{ptwise_convo_bound}) as
$$ \Bigl| \HF_k(x-y) - \sum_{\xi\in\Xi} a(\xi,y) \HF_k(x-\xi) \Bigr| \le C2^{kN}(h2^{k})^{\omega} (1+2^k|x-y|)^{-(N+1)}.$$
Integrating over $\R^N$ with respect to $x$ gives
\begin{eqnarray*}
\|(\tilde{\nu}-\tilde{\nu}_{\Xi})*\HF_k\|_{L_1(\R^N)}  
&\le&
\int_{\R^N} \int_{\Omega} |\nu(y) | 
\left| \HF_k(x-y) - \sum_{\xi\in\Xi} a(\xi,y) \HF_k(x-\xi) \right| \diff x \diff\sigma(y)\\
&\le& C  \|\nu\|_{L_{1}(\sigma)} \sup_{y\in \Omega} 
\int_{\R^N}  2^{kN}  (h2^{k})^{\omega} (1+ 2^k {|x-y|})^{-(N+1)}\diff x \\
&\le&
C (h2^{k})^{\omega} \|\nu\|_{L_{1}(\sigma)} .
\end{eqnarray*}
Integrating over $\Omega$ with respect to $y$ gives, 
after applying (\ref{integral_test}) with 
$a =2^k$, the estimate
\begin{eqnarray*}
\|(\tilde{\nu}-\tilde{\nu}_{\Xi})*\HF_k\|_{L_{\infty}(\R^N)} 
&\le & \sup_{x\in \R^N}
\int_{\Omega} |\nu(y) | 
\left| \HF_k(x-y) - \sum_{\xi\in\Xi} a(\xi,y) \HF_k(x-\xi) \right| \diff\sigma(y)\\
&\le& 
C 2^{kN} (h2^k)^{\omega}\|\nu\|_{L_{\infty}(\sigma)} \sup_{x\in \R^N} \int_{\Omega}  (1+ 2^k {|x-y|})^{-(N+1)}\diff \sigma(x) \\
&\le&
C  2^{k(N-d)}(h2^k)^{\omega}\|\nu\|_{L_{\infty}(\sigma)} .
\end{eqnarray*}
The estimate
$ \|(\tilde{\nu}-\tilde{\nu}_{\Xi})*\HF_k\|_{L_{p}(\R^N)}
\le C 2^{k\frac{N-d}{p'} } (h2^k)^{\omega} \|\nu\|_{L_{p}(\sigma)} 
$
 for $Kh<2^{-k}$ follows by interpolating between $p=1$ and $p=\infty$  and by applying  Riesz-Thorin 
 to the linear map $\nu\mapsto (\nu-\tilde{\nu}_{\Xi})*\HF_k$  .

 \medskip
 
{\em Case 2:} If $2Kh>2^{-k}$,
the estimate (\ref{ptwise_convo_bound}) is enough to produce the $L_{\infty}(\R^N)$ bound. It works as before.
We can control
$\|(\tilde{\nu}-\tilde{\nu}_{\Xi})*\HF_k\|_{L_{\infty}(\R^N)} $ 
by integrating over $\Omega$,
obtaining the expression
\begin{multline*}
\|(\tilde{\nu}-\tilde{\nu}_{\Xi})*\HF_k\|_{L_{\infty}(\R^N)}
\le\\
C 2^{kN}  \|\nu\|_{L_{\infty}(\sigma)}
 \sup_{x\in \R^N}  
  \left(
\sigma\bigl(B_{\Omega}(y,2Kh)\bigr)
+
\int_{\{y\in\Omega \mid2Kh<|y-x|\}}
(2^k {|x-y|})^{-(N+1)}\diff \sigma(x) 
\right).
\end{multline*}
The measure of the ball can be estimated by the 
assumption on $\sigma$, while the integral can 
be estimated by a series test, similar to (\ref{integral_test}):
decompose $\{x\in\Omega\mid |x-y|>2Kh\}=\bigcup_{j=1}^\infty A_j$
en annuli, with
$A_j := \{x\in \Omega \mid  2Kh 2^{j-1} \le |x-y|\le 2Kh 2^j\}$.
Then $x\in A_j$ satisfies $2^k |x-y| \ge 2^{j-1} 2Kh 2^k$,
so 
$$
\int_{\{y\in\Omega \mid2Kh<|y-x|\}}
(2^k {|x-y|})^{-(N+1)}\diff \sigma(x) 
\le \omega_{\Omega} (2Kh )^d (2Kh 2^k)^{-(N+1)}  \sum_{j=1}^{\infty} (2^{j-1})^{-(N+1)}2^{jd}.$$
The last series is convergent, since $N\ge d$, so we have
\begin{eqnarray}
\|(\tilde{\nu}-\tilde{\nu}_{\Xi})*\HF_k\|_{L_{\infty}(\R^N)} 
&\le &
C2^{k N}    \|\nu\|_{L_{\infty}(\sigma)} 
\bigl(
(2Kh)^d+(2Kh2^k)^{d-(N+1)}
\bigr) 
\nonumber\\
&\le& C2^{k N}   h^{d}  \|\nu\|_{L_{\infty)}(\sigma)} .
\label{inf_mix_scale}
\end{eqnarray}

To handle the $L_1(\R^N)$ norm, consider 
$\int_{\Omega}\nu(y) \psi_k(x-y)\diff \sigma(y)$
 and
$\sum  \int_{\Omega} \nu(y) a(\xi,y) \diff y \psi_k(x-y)$ 
separately.
For the first term, 
$\| \int_{\Omega} \nu(y) \HF_k(\cdot-y)\diff \sigma(y)\|_{L_1(\R^N) }\le C \|\nu\|_{L_1(\sigma)} $,
since a change of variable ensures
 $C=\|\HF_k\|_{L_1(\R^N)}=\|\HF\|_{L_1(\R^N)}$.
 Similarly, $\|\psi_k(x-\xi)\|_{L_1(\R^N)} \le C$, so we have
 $$\left\| 
 \Bigl(\sum_{\xi\in\Xi}  \int_{\Omega} \nu(y) a(\xi,y) \diff y \Bigr) \psi_k(\cdot-\xi)\right\|_{L_1(\R^N)}
 \le C\int_{\Omega}| \nu(y)| \sum |a(\xi,y) | \diff \sigma(y)\le C\|\nu\|_{L_1(\sigma)}.$$
Thus $\|(\nu-\tilde{\nu}_{\Xi})*\HF_k\|_{L_{1}(\R^N)}
\le \|\nu*\HF_k\|_{L_{1}(\R^N)} +\|\tilde{\nu}_{\Xi}*\HF_k\|_{L_{1}(\R^N)}\le C\|\nu\|_{L_1(\sigma)}$.

Interpolating between $p=1$ and the  $p=\infty$ estimate (\ref{inf_mix_scale}) gives the estimate  
$\|\tilde{\nu}_{\Xi}*\HF_k\|_{L_{p}(\R^N)} 
\le
C
  2^{\frac{kN}{p'} }   (h)^{\frac{d}{p'}}  \|\nu\|_{L_p(\sigma)}.
$
when $2Kh>2^{-k}$.
\end{proof}

This permits us to measure the approximation error $\nu-\tilde{\nu}_{\Xi}$ in the relevant Besov spaces $B_{p,q}^s(\R^N)$.
\begin{lemma}
\label{L_measure_approx}
Suppose $L$ is a positive integer
and suppose
$\Omega$ and $\sigma$ satisfy assumption \ref{ball_comparison}.
Then for any $p,q\in [1,\infty]$ and $s\in\R$ satisfying
$$\Bigl(\frac{d-N}{p'}-L,1\Bigr) \prec(s,q)\preceq \Bigl(- \frac{N}{p'},\infty\Bigr)$$  
there is a $C$
so  that if $a:\Xi\times \Omega\to \R$ is a 
local polynomial reproduction of degree 
$L-1$ (i.e., it satisfies Assumption \ref{A_LPR})
and  $\nu\in L_p(\sigma)$, then
$$\|\tilde{\nu}- \tilde{\nu}_{\Xi}\|_{B_{p,q}^s(\R^N)} \le C h^{ (d-N)/p' - s} \|\nu\|_{L_p(\sigma)}.$$
\end{lemma}
\begin{proof}
The condition that $L\ge (d-N)/p'-s$ ensures
$$
\| \LF*(\tilde{\nu}- \tilde{\nu}_{\Xi})\|_p
\le 
C h^{L}
 \|\nu\|_{L_p(\sigma)}
\le 
C h^{(d-N)/p' - s}  \|\nu\|_{L_p(\sigma)}
$$
for every $s$. Thus it suffices to treat the portion of the Besov
norm determined by the seminorms 
$\|\HF_k*(\tilde{\nu}-\tilde{\nu}_{\Xi})\|_p$.

We consider two cases: the strongest norm $B_{p,\infty}^{-N/p'}(\R^N)$
and the weaker norms $B_{p,q}^{s}(\R^N)$ with $(s,q)\precneqq (-N/p', \infty)$. The latter case
can be handled by considering $q=1$, since the continuous embedding $B_{p,1}^{s}(\R^N) \subset B_{p,q}^{s}(\R^N)$ holds
and the error estimate does not depend on $q$.

{\em Case 1} In this case $(s,q)= (-N/p', \infty)$.
We consider 
$\sup_{k\in \N} 2^{-kN/p'} \|\HF_k*(\tilde{\nu}-\tilde{\nu}_{\Xi})\|_p$, which we 
split over the initial indices $k\le  |\log_2 (2K h)|$ 
and the tail $k>|\log_2 (2K h)|$. 
For the tail,
we have
\begin{eqnarray*}
\sup_{k >|\log_2 (2Kh)|} 2^{-kN/p'} \|\HF_k*(\tilde{\nu}-\tilde{\nu}_{\Xi})\|_p
&\le& 
C
\|\nu\|_{L_p(\sigma)}
\sup_{k >|\log_2 (2Kh)|} 
2^{-kN/p'}  2^{kN/p'} h^{d/p'} \\
&\le& C\|\nu\|_{L_p(\sigma)}  h^{d/p'}  .
\end{eqnarray*}
For the initial part, we use  $\omega = d/p'$,
which gives
$2^{k(N-d)/p'} (h2^k)^{d/p'} = 2^{kN/p'} h^{d/p'}$.
Thus we have
$$\sup_{k \le |\log_2 (2Kh)|} 2^{-kN/p'} \|\HF_k*(\tilde{\nu}-\tilde{\nu}_{\Xi})\|_p
\le C\|\nu\|_{L_p(\sigma)}  h^{d/p'}  $$
as well.

\medskip

{\em Case 2} In this case $s<-N/p'$ and we consider 
$\sum_{k\in \N} 2^{ks } \|\HF_k*(\tilde{\nu}-\tilde{\nu}_{\Xi})\|_p$. Again, we
split this over the initial partial sum with $k\le |\log_2(2K h)| $
and the tail $k>|\log_2 (2Kh)| $. 

Since $s< -N/p'$, the tail is a convergent geometric series,
with 
\begin{eqnarray*}
\sum_{k>|\log_2 h| } 2^{ ks} \|\HF_k*(\tilde{\nu}-\tilde{\nu}_{\Xi})\|_p
&\le& C\|\nu\|_{L_p(\sigma)} \sum_{k>|\log_2 h| } 2^{ k(s +N/p')} h^{d/p'} \\
&\le& C h^{d/p'}\|\nu\|_{L_p(\sigma)} 2^{(s +N/p')\log_2(h)}.
\end{eqnarray*}
After simplifying, this gives
$\sum_{k>|\log_2 h| } 2^{ ks} \|\HF_k*(\tilde{\nu}-\tilde{\nu}_{\Xi})\|_p\le C\|\nu\|_{L_p(\sigma)}h^{(d-N)/p'  -s}$.

In the initial partial sum, we consider 
$  (d-N)/p' -s< L $, which yields a geometric series with positive exponent
\begin{eqnarray*}
\sum_{k\le |\log_2 (2Kh)| } 2^{ ks} \|\HF_k*(\tilde{\nu}-\tilde{\nu}_{\Xi})\|_p
&\le& C\|\nu\|_{L_p(\sigma)} \sum_{k\le |\log_2 (2K h)| } 2^{ k(s +(N-d)/p')} (h2^k)^{L} \\
&\le& 
C h^{L}\|\nu\|_{L_p(\sigma)} 2^{(s +(N-d)/p'+\omega)|\log_2(2Kh)|}\\
&\le& C  \|\nu\|_{L_p(\sigma)} h^{-s -(N-d)/p')}.
\end{eqnarray*}
The lemma follows by combining this with the estimate on the tail.
\end{proof}

\paragraph{Funding statement}
The work of GBW was partially supported by US National Science Foundation grants 2309712 and 2505987.

\paragraph{Data availability statement}
The point sets used in the numerical experiments are available from the \texttt{spherepts} package \url{https://github.com/gradywright/spherepts}.  The code for constructing the differentiation matrices for the numerical experiments from Section \ref{sec:egvl_numerics} is available from the KernelDMSuite package \url{https://github.com/gradywright/kerneldmsuite}.

\bibliographystyle{plain}
\bibliography{literature}

\end{document}